\documentclass[10pt]{amsart}
\usepackage{amsmath, amsthm}
\usepackage{eucal}

\usepackage{amssymb}
\usepackage{amscd}
\usepackage{palatino}
\usepackage{latexsym}
\usepackage{epsfig}
\usepackage{graphicx}
\usepackage{amsfonts}
\usepackage{psfrag}

\usepackage{url}
\usepackage{cite}

\usepackage{hyperref}
\input xy
\xyoption{all}
\UseComputerModernTips

\numberwithin{equation}{section}
\numberwithin{figure}{section}

\newtheorem{theorem}{Theorem}[section]
\newtheorem{lemma}[theorem]{Lemma}
\newtheorem{proposition}[theorem]{Proposition}

\newtheorem{remark}[theorem]{Remark}

\theoremstyle{definition}
\newtheorem{definition}[theorem]{Definition}

\newcommand{\C}{{\mathbb{C}}}
\renewcommand{\H}{{\mathbb{H}}}
\newcommand{\Z}{{\mathbb{Z}}}

\newcommand{\R}{{\mathbb{R}}}
\newcommand{\N}{{\mathbb{N}}}
\newcommand{\T}{{\mathbb{T}}}
\renewcommand{\L}{L}  
\newcommand{\p}{\mathbb{P}}
\renewcommand{\P}{{\mathcal{P}}}

\renewcommand{\t}{\mathfrak{t}}

\newcommand{\into}{\hookrightarrow}

\newcommand{\Flags}{{\mathcal F}\ell ags}

\renewcommand{\Re}{\operatorname{Re}}
\renewcommand{\Im}{\operatorname{Im}}

\DeclareMathOperator{\U}{U}

\DeclareMathOperator{\vol}{vol}
\DeclareMathOperator{\id}{id}
\newcommand{\abs}[1]{\lVert#1\rVert}
\newcommand{\sabs}[1]{\lvert#1\rvert}  
  
\newcommand{\sAbs}[1]{\bigl\lvert#1\bigr\rvert}  
\newcommand{\Abs}[1]{\bigl\lVert#1\bigr\rVert}
\newcommand{\aabs}[1]{\left\lvert#1\right\rvert}
\newcommand{\inv}{^{-1}} 
\newcommand{\define}[1]{\textsl{#1}}
\newcommand{\cross}{\times}
\newcommand{\st}{\mid}  

\newcommand{\del}{\partial}
\newcommand{\dd}[2]{\frac{\del #1}{\del #2}}

\newcommand{\hsm}{{\hspace{1mm}}}

\newcommand{\X}{{\mathcal X}}
\newcommand{\K}{\mathcal{K}}
\renewcommand{\U}{\mathcal{U}}

\newcommand{\kahler}{K\"{a}hler}     
  
\newcommand{\restr}[1]{\bigr\rvert_{#1}} % in ../head.tex

\begin{document}

\title{Convergence of polarizations, toric degenerations, and
  Newton-Okounkov bodies}

\author{Mark Hamilton}
\address{Department of Math and Computer Science\\ 
Mount Allison University \\ 67 York St.\\ Sackville, NB, E4L 1E6\\ Canada}
\email{mhamilton@mta.ca}
%\urladdr{\url{...}}
\thanks{The first author is partially supported by an NSERC Discovery
  Grant.} 

\author{Megumi Harada}
\address{Department of Mathematics and
Statistics\\ McMaster University\\ 1280 Main Street West\\ Hamilton, Ontario L8S4K1\\ Canada}
\email{Megumi.Harada@math.mcmaster.ca}
\urladdr{\url{http://www.math.mcmaster.ca/Megumi.Harada/}}
\thanks{The second author is partially supported by an NSERC Discovery Grant
and a Canada Research Chair (Tier 2) Award.} 

\author{Kiumars Kaveh}
\address{Department of Mathematics\\ University of Pittsburgh\\ 301 Thackeray Hall\\
Pittsburgh, PA, 15260\\ USA} 
\email{kaveh@pitt.edu} 
\urladdr{\url{http://www.pitt.edu/~kaveh/}} 

\thanks{The third author is partially supported by a National Science Foundation Grant 
(Grant ID: DMS-1601303), Simons Foundation Collaboration Grant for Mathematicians, and Simons Fellowship award.}


\date{\today}

%%%%%%%%%%%%%%%%%%%%%
%  Abstract
%%%%%%%%%%%%%%%%%%%%%

\begin{abstract}

  Let $X$ be a smooth irreducible complex algebraic variety of
  dimension $n$ and $L$ a very ample line bundle on $X$. Given a toric
  degeneration of $(X,L)$ satisfying some natural technical
  hypotheses, we construct a deformation $\{J_s\}$ of the complex
  structure on $X$ and bases $\mathcal{B}_s$ of $H^0(X,L, J_s)$ so
  that $J_0$ is the standard complex structure and, in the limit as
  $s \to \infty$, the basis elements approach dirac-delta
  distributions centered at Bohr-Sommerfeld fibers of a moment map
  associated to $X$ and its toric degeneration. The theory of
  Newton-Okounkov bodies and its associated toric degenerations shows
  that the technical hypotheses mentioned above hold in some
  generality.  Our results significantly generalize previous results
  in geometric quantization which prove ``independence of
  polarization'' between K\"ahler quantizations and real
  polarizations.  As an example, in the case of general flag varieties
  $X=G/B$ and for certain choices of $\lambda$, our result
  geometrically constructs a continuous degeneration of the
    (dual) canonical basis of $V_{\lambda}^*$ to a collection of
  dirac delta functions supported at the Bohr-Sommerfeld fibres
  corresponding exactly to the lattice points of a
  Littelmann-Berenstein-Zelevinsky string polytope
  $\Delta_{\underline{w}_0}(\lambda) \cap \Z^{\dim(G/B)}$.
\end{abstract}

\maketitle

%
%\begin{center} \textit{This is a preliminary version. Comments are
%    welcome.}  \end{center}

\setcounter{tocdepth}{1}
\tableofcontents

\section{Introduction}
The motivation for the present manuscript arose from two rather
different research areas: the theory of geometric quantization in
symplectic geometry on the one hand, and the algebraic-geometric theory of
Newton-Okounkov bodies - particularly in its relation to
representation theory - on the other. Since we do not expect 
all readers of this paper to be familiar with both theories, we begin
with a brief description of each. 

We begin with a sketch of geometric quantization. As is well-known, symplectic geometry (Hamiltonian flows on
symplectic manifolds) is the mathematical language 
for formulating classical physics, whereas it is the language of
linear algebra and representation theory (unitary flows on Hilbert
spaces) which forms the basis for formulating quantum physics. It has
been a long-standing question within symplectic geometry to
understand, from a purely mathematical and geometric perspective, the relation between
the classical picture and the quantum picture, in terms of both the
phase spaces and the defining equations of the dynamics. In one
direction, to go from ``quantum'' to ``classical'', one can ``take a
classical limit''. The reverse direction, i.e. that of systematically
associating to a symplectic manifold $(M,\omega)$ a Hilbert space
$Q(M,\omega)$ and to similarly relate, for instance, Hamilton's
equations on $(M,\omega)$ to Schr\"odinger-type equations on
$Q(M,\omega)$, is generally referred to as the theory of
\emph{quantization}. In this manuscript, we deal specifically
with \emph{geometric quantization}, a theory which associates
to a symplectic manifold $(M,\omega)$ a Hilbert space $Q(M,\omega)$. 

For a fixed $(M,\omega)$, it turns out that there are many possible
ways of constructing a suitable Hilbert space $Q(M,\omega)$. To describe the choices we first
set some notation. First suppose that $[\omega]$ is an integral
cohomology class.
Next, let $(L,\nabla, h)$ be a Hermitian line bundle with
connection satisfying $\mathrm{curv}(\nabla)=\omega$. Such a triple is called
a \emph{pre-quantum line bundle}, or sometimes a
\emph{pre-quantization}.
Also required is a \emph{polarization,} of which the two main
types are as follows. 
A \emph{K\"ahler polarization}
is a choice of compatible complex structure $J$ on
$M$. Given such a $J$, one can define the quantization $Q(M, \omega)$
to be $H^0(M, L, J)$, the space of holomorphic sections of $L$ with
respect to this complex structure $J$. On the other hand, one may also
consider a (possibly singular) \emph{real polarization} of $M$,
which is a foliation of $M$ into Lagrangian submanifolds. Among the
Lagrangian leaves one can define a special
(usually finite, if $M$ is compact) subset
called the \emph{Bohr-Sommerfeld leaves}. There is not yet an agreed-upon
``correct'' definition of the corresponding Hilbert space for a real
polarization, but one approach which has been investigated, and which
will be used in this manuscript, is to consider distributional
sections supported on the set of Bohr-Sommerfeld leaves. Based on the
above discussion, the following natural question arises: \emph{Is the quantization
  $Q(M,\omega)$ ``independent of polarization,'' i.e., independent of
  the choices made?} More specifically, we can ask: does the quantization
coming from a K\"ahler polarization agree with the quantization coming from a
real polarization?  The results of this manuscript confirms 
independence of polarization in a rather large class of examples,
significantly extending previously known results which were restricted
to special cases such as toric varieties and flag varieties.

We next briefly motivate the theory of Newton-Okounkov bodies. 
The famous Atiyah-Guillemin-Sternberg and Kirwan convexity theorems link
equivariant symplectic and algebraic geometry to the combinatorics of
polytopes. In the case of a toric variety $X$, the combinatorics of
its moment map polytope $\Delta$ fully encodes the geometry of $X$,
but this fails in the general case. In his influential work, Okounkov
constructed (circa roughly 1996), for an (irreducible) projective
variety \(X \subseteq \p(V)\) equipped with an action of a reductive
algebraic group $G$, a convex body $\tilde{\Delta}$ and a
natural projection from $\tilde{\Delta}$ to the moment polytope
$\Delta$ of $X$. Moreover, the volumes of the fibers of this
projection encode the asymptotics of the multiplicities of the irreducible
representations appearing in the homogeneous coordinate ring of $X$,
or in other words, the Duistermaat-Heckman measure \cite{Okounkov-BM, Okounkov-log-concave}. Recently, Askold
Khovanskii and the third author (also independently Lazarsfeld and
Mustata \cite{LazMus}) 
vastly generalized Okounkov's ideas \cite{KavKho}, and in
particular constructed such $\tilde{\Delta}$ (called
\textit{Newton-Okounkov bodies} or sometimes simply \textit{Okounkov
  bodies}) even without presence of any group action. In the setting
studied by Okounkov, the maximum possible (real) dimension of the
Newton-Okounkov body $\tilde{\Delta}$ is the
  transcendence degree of $\C(X)^U$ where $U$ is a maximal unipotent
  subgroup of $G$; when there is no group action (as in the setting
  studied in \cite{LazMus, KavKho}) 
we have 
\(\dim_{\R}(\tilde{\Delta}) = \dim_{\C}(X).\) 
Hence one interpretation of the results of Okounkov, Lazarsfeld-Mustata and Kaveh-Khovanskii is that there \emph{is} a convex geometric/combinatorial object of `maximal'
dimension associated to $X$, even when $X$ is not a toric
variety. This represents a vast expansion of the possible settings in
which combinatorial methods may be used to analyze the geometry of
algebraic varieties. 
There is promise of a rich theory which interacts with a wide range of
inter-related 
areas: for instance, the third author showed \cite{Kav-cry} 
that the \textit{Littelmann-Berenstein-Zelevinsky string polytopes} from
representation theory, which generalize the well-known
\textit{Gel'fand-Cetlin polytopes}, are examples of $\tilde{\Delta}$. In the
long-term, one can expect further applications to Schubert calculus
and to geometric representation theory 
(e.g. see \cite{KST}). 

We now turn attention to the present manuscript. 
Firstly we should explain that the two seemingly disparate research
areas mentioned above are related due to the results
in \cite{HarKav}, which uses a certain toric 
degeneration that arises from (the semigroup associated to) a
Newton-Okounkov body \cite{Anderson} to construct \emph{integrable systems}\footnote{Here we use the term
  ``integrable system'' in the slightly non-standard sense of
  \cite[Definition 1.1]{HarKav}.}  on a wide class
of projective varieties. Integrable systems are highly special
Hamiltonian systems on symplectic (or, in our setting, K\"ahler)
manifolds, and naturally give rise to (singular) real polarizations. Therefore,
the theory of Newton-Okounkov bodies and their associated toric
degenerations provide a natural setting in which to examine the theory
of geometric quantization.

Before describing the statement of our main result (Theorem~\ref{theorem:main}) in more
detail we first recall the content of two manuscripts of Baier, Florentino, Mourao, and Nunes \cite{BFMN} and 
the first author and Konno 
\cite{HamKon}, on which much of the current manuscript is based. 

As already mentioned, a natural question that arises in the theory of
geometric quantization is that of independence of polarization, i.e.,
the isomorphism class of a geometric quantization should be
independent of the choices made. In the above context, this means we
wish to show $\dim H^0(X, L, J)$ is equal to the number of
Bohr-Sommerfeld fibres.  Nunes and his collaborators initiated a
``convergence of polarizations'' approach to this question.
Specifically, they deform the complex structure on $X$ in such a way
that the \kahler\ polarization it defines converges, in a suitable
sense to be further explained in Section~\ref{sec-main-result}, to the
real polarization on the same manifold.  (See~\cite{nunes} 
for an
overview of this program.)  Although there are more general versions
of this theory, in this paper we focus particularly on the case of symplectic
toric manifolds as described in~\cite{BFMN}, where the \kahler\
polarization converges to the (singular) real polarization given by
the fibres of the moment map (i.e., the integrable system) for a torus
action.  Indeed, for the case of a symplectic toric manifold $X$
associated to a Delzant polytope $\Delta$, it is well-known (see for
example~\cite{Ham-toric}) that there is a natural basis
$\{ \sigma^m \st m\in \Delta\cap \Z^n\}$ of the space $H^0(X,L, J)$ of
holomorphic sections of $L$ that is indexed by the integer lattice
points in $\Delta$.  It is also well-known that the Bohr-Sommerfeld
fibres in $X$ are exactly the moment map fibres over precisely the
same set of integer lattice points $\Delta \cap \Z^n$. In particular,
the dimensions of the two quantizations agree.  This is often seen as
one of the most basic and motivational examples of the phenomenon of
``independence of polarization''.

The first author and Konno extend the results of Baier, Florentino,
Mourao, and Nunes \cite{BFMN}, which only apply to toric manifolds, to
the case of the complete flag variety $\Flags(\C^n)$ by making use of
a toric degeneration of $\Flags(\C^n)$ as constructed in
\cite{Kogan-Miller}. The precise definition of a toric degeneration is
given in Section~\ref{sec-main-result}; roughly, it is a (flat) family
of algebraic varieties over $\C$ with generic fiber isomorphic to a
given variety (in this case $\Flags(\C^n)$) and special fiber a toric
variety.  The gradient-Hamiltonian-flow technique pioneered by Ruan
\cite{Ruan} allows one to ``pull back'' the integrable system on the
special fiber to one on the original variety $\Flags(\C^n)$ and also
enables the authors to apply the techniques of \cite{BFMN} to
$\Flags(\C^n)$.
 
With the above as motivation, we now describe
the main result of this manuscript, although we do not give
the full and precise statement due to its rather technical nature. Let
$X$ be a smooth, irreducible complex
algebraic variety with $\dim_{\C}(X) = n$, equipped with prequantum
data $(L, \nabla, h)$. Suppose $X$ admits a toric degeneration $\X$
as above (and made precise in Section~\ref{sec-main-result}). Under these assumptions, we can construct from the
toric degeneration $\X$ an integrable system $\mu: X \to \R^n$ on $X$
as in \cite{HarKav}. Very roughly, the main result of this paper then states
the following (see Theorem~\ref{theorem:main} for the precise
statement).

\medskip
\noindent \textbf{Theorem A.} 
  Under some natural technical hypotheses on $X$ and its toric
  degeneration $\X$, there exists a continuous deformation $\{J_s\}_{s
    \in [0,\infty)}$ of the complex structure on (the underlying
  smooth manifold of) $X$ such that $J_0$ is the original complex
  structure on $X$, and in the limit as $s$ goes to $\infty$, the
  K\"ahler polarization defined by $J_s$
  converges to the (singular) real polarization
  associated to the integrable system $\mu$ on $X$.

\medskip
A notable feature of the above theorem is that, following the work of \cite{BFMN, HamKon}, 
Theorem A gives an explicit correspondence between specific elements of the K\"ahler and the real quantization (rather than just an equality of dimensions). The theorem above places additional hypotheses on $X$ and its toric
degeneration, so one immediately then asks: when do these hypotheses
hold? The second result of this manuscript, made precise in
Theorem~\ref{theorem:toric deg from NOBY}, 
is that the construction given in
\cite{HarKav} gives a large class of examples on which Theorem
A applies, with the caveat that it is necessary to replace the
original line bundle $L$ with a suitable tensor power (or,
equiv galently, the original symplectic form with a positive integer
multiple thereof). 
Roughly, our result (Theorem~\ref{theorem:toric deg from NOBY}) states the following. 

\medskip
\noindent \textbf{Theorem B.} 
Let $X$ be as above, equipped with the line bundle $L$. Then the construction of the toric degeneration
given by valuations (as in \cite{HarKav}) can be made to satisfy the additional
technical hypotheses in Theorem A for the pair $(X, L^{\otimes d})$
for sufficiently large $d$ and thus gives `convergence of polarization' in these cases.
\medskip

It is also worth mentioning that, in the precise statement of our main Theorem A, as given in Theorem~\ref{theorem:main} below, we assert an \emph{existence} of a certain basis of holomorphic sections with appropriate convergence properties, which then implies Theorem A. In the general case considered in Theorem A, this basis is not very explicit for some of the values of the deformation parameter. However, in the situation of Theorem B where the toric degeneration arises from a valuation as above, we additionally show in Theorem~\ref{thm-sections-lin-ind} that this basis can be chosen to be both natural and explicit throughout the deformation. 

As already mentioned, there are several indications of interesting
connections between the theory of Newton-Okounkov bodies and
representation theory. Indeed, putting the results of \cite{Kav-cry}
and \cite{HarKav} together, we obtain an integrable system on a flag
variety $G/B$ whose moment map image is precisely the
Littelmann-Berenstein-Zelevinsky string polytope
$\Delta_{\underline{w}_0}(\lambda)$. This construction uses the
so-called \emph{(dual) canonical basis} of $V_{\lambda}^*$, the dual
space of the $G$-module $V_\lambda$ with highest weight $\lambda$. The
elements of this basis are parametrized by the lattice points
$\Delta_{\underline{w}_0}(\lambda) \cap \Z^{\dim(G/B)}$.  The
integrable system gives rise to a real polarization of
$L_\lambda \to G/B$ whose Bohr-Sommerfeld fibers are in one-to-one
correspondence with
$\Delta_{\underline{w}_0}(\lambda) \cap \Z^{\dim(G/B)}$, where
$L_\lambda$ is the usual pullback line bundle from the Pl\"ucker
embedding associated to weight $\lambda$.  Moreover, in this context,
the Borel-Weil theorem implies that the K\"ahler quantization
$H^0(G/B, L_\lambda)$ can be identified with $V_{\lambda}^*$. Hence,
in this special case and for sufficiently large multiples of
  $\lambda$, our Theorem A geometrically constructs a continuous
degeneration of a basis of $V_{\lambda}^*$ to a
collection of dirac delta functions supported at the Bohr-Sommerfeld
fibres corresponding exactly to the lattice points
$\Delta_{\underline{w}_0}(\lambda) \cap \Z^{\dim(G/B)}$.

The paper is organized as follows. In Section~\ref{sec-main-result} we
recall the necessary definitions and give a full and precise statement
of our main Theorem~\ref{theorem:main}. We set up the necessary
family of complex structures, based largely on the work of \cite{BFMN} and
\cite{HamKon}, in Section~\ref{sec:variation}. The proof of
Theorem~\ref{theorem:main} occupies Section~\ref{sec:proof}. We then
show in Section~\ref{sec:NOBY} that the construction in \cite{HarKav}
gives many examples of toric degenerations satisyfing the hypotheses
of Theorem~\ref{theorem:main}.

We close with some brief comments on open questions. 
Firstly, we believe that the proof of our main results can be
modified to work for an embedding of the toric degeneration into $Y
\times \C$ where $Y$ is any smooth projective toric variety (instead
of just a projective space). Secondly, we also believe that the
constructions in this paper should descend to a GIT quotient by a
torus action. We leave these open for future work.

\bigskip
\noindent \textbf{Acknowledgements.} We thank Yael Karshon for
providing the opportunity for us to learn about each other's past work,
thus helping to initiate this collaboration.

\section{Statement of the main theorem} \label{sec-main-result}

This section is devoted to the full and precise formulation of both
the hypotheses for, and the statement of, our main theorem. We first
provide a quick overview of geometric quantization and then dive
straight into the technicalities of our theorem. Some key motivational
remarks, which may aid a reader unfamiliar with this material, are
contained 
in Remark~\ref{remark:motivation}.

We begin with definitions in the theory of
geometric quantization. For details see 
e.g. \cite{Woodhouse}. 
Let $(X,\omega)$ be a symplectic manifold, i.e. $X$ is a smooth
manifold and $\omega$ is a closed
non-degenerate differential $2$-form on $X$. Suppose 
$(\L, h, \nabla)$ is a complex line bundle with Hermitian structure $h$
and a connection $\nabla$ satisfying 
$\text{curv}(\nabla) = \omega$ and with parallel transport preserving $h$. Such a triple is called a
\define{prequantum line bundle} (or sometimes \emph{prequantum data},
or \emph{prequantization}) of $(X,\omega)$. Note that for a prequantum line
bundle to exist, $[\omega]$ must be an integral cohomology class.

To pass from a prequantization to a quantization, we must choose a
\emph{polarization,} which is an integrable
complex Lagrangian distribution on $X$. 
We only deal with two types of
polarizations in this manuscript, as follows. 
Firstly, a \define{K\"ahler polarization} on $(X,\omega)$ is a 
compatible complex structure $J$. Then 
the corresponding \define{K\"ahler quantization} of $X$ 
is defined to be 
the space of holomorphic sections $H^0(X,L)$ of $L$, where 
$L$ is equipped with the holomorphic structure specified by $J$.  
Secondly, a \define{(singular) real polarization} on $X$ is a 
(singular) foliation of $X$ into Lagrangian submanifolds.  
Let $P$ denote the distribution in $TX$ corresponding to a real
polarization. By abuse of language, we frequently refer to
both the foliation and the distribution as a real polarization. 
A special case, of 
much recent interest in this area, 
is the (singular) foliation given by the fibres of a completely 
integrable system $F\colon X \to \R^n$. 
In this setting, a 
section $\sigma$ of $\L\restr{U}$ over some open set $U\subset X$ is 
said to be \define{flat along the leaves} or \define{leafwise flat} 
if it is covariantly constant with respect to $\nabla$ in directions 
tangent to $P$.
Leafwise flat sections always exist locally, but not usually globally.  
A leaf $\ell$ of the real polarization $P$ is a \define{Bohr-Sommerfeld leaf} 
if there exists a (nonzero) section $\sigma$ that is flat along the leaves 
and defined on all of $\ell$. 
The set of Bohr-Sommerfeld leaves is typically discrete in the space of 
leaves. 

There is not at present a single agreed-upon definition of
quantization using a real polarization.  
The basic philosophy is that the quantization corresponding to a real
polarization ``should'' be given by
leafwise flat sections over Bohr-Sommerfeld fibres, but since there
are no globally defined leafwise flat sections (the set of
Bohr-Sommerfeld fibers being usually discrete), this is not
straightforward. 
One possible approach is to relax the requirement that the sections be smooth 
and look at distributional sections supported on the set of 
Bohr-Sommerfeld leaves.  
Several examples have been 
investigated using this approach; see~\cite{nunes} and references 
therein. This is also the approach we take in this manuscript.

We now set up the terminology and notation
required for the statement of our main theorem. 
Let $X$ be a smooth, irreducible complex algebraic variety with
$\dim_\C(X)=n$. We suppose in addition that $X$ is equipped with
prequantum data $(\omega, J, L, h, \nabla)$ as above, where
$(\omega, J)$ is a K\"ahler structure on $X$ and $L$ is a very ample line
bundle over $X$ with Chern class equal to the K\"ahler class
(i.e. $c_1(L)=[\omega] \in
H^2(X,\Z)$).
In \cite{HarKav}, using ideas from \cite{NNU}, a toric degeneration is used to construct
an integrable system on $X$ which is a Hamiltonian $T^n$-action on an
open dense subset of $X$. Recall that a \emph{toric degeneration} of $X$ in
the sense of \cite{HarKav} is a flat family $\pi: \X \to \C$
of irreducible varieties such that the family is trivial over
$\C^*= \C\setminus \{0\}$ with each fiber isomorphic to $X$, and the
(possibly singular) central fiber $X_0 := \pi^{-1}(0)$ is a toric variety with respect to
a complex torus $\T_0$. 
In particular there exists a
fiber-preserving isomorphism $\varrho: X \times \C^* \to \pi^{-1}(\C^*)$
from the trivial fiber bundle $X \times \C^* \to \C^*$ to
$\pi^{-1}(\C^*)$ and it follows by assumption on $X$ that $\X$ is
smooth away from $X_0$. For a fixed $t \in \C^*$, let $X_t := \pi^{-1}(t)$
denote the fiber of the family $\X$ and let $\varrho_t$ denote the
restriction of $\varrho$ to $X \times \{t\}$. 
By assumption, $\varrho_1$ is an
isomorphism from $X \cong X\times \{t\}$ to $X_1$. We
will frequently identify $X$ with $X_1$ using this isomorphism
$\varrho_1$. 

In this manuscript, following \cite{HarKav} we assume that $X$ admits a toric degeneration
with the additional property that
the family $\X$ can be embedded in $\P
\times \C$ (where $\P \cong \p^N$ is a projective space for an
appropriate choice of $N$), as an
algebraic subvariety such that 
\begin{itemize}
\item[(a)] the map $\pi: \X \to \C$ is the restriction to $\X$ of the
  usual projection $\P \times \C \to \C$ to the second factor, and  
\item[(b)] the action of $\T_0$ on $X_0$ extends to a linear action of
  $\T_0$ on $\P \cong \P \times\{0\}$. 
\end{itemize}
Sometimes by abuse of notation we think of $X_t \subseteq \P \times \{t\}$ as a subvariety
of $\P$ via the natural identification $\P \cong \P \times \{t\}$. 
Next we equip the ambient projective space $\P$
with prequantum data $(\omega_\P,
J_\P, L_\P \cong \mathcal{O}(1), \nabla_\P, h_\P)$. 
In addition to the prequantum data on $\P$, we need data
on $\P \times \C$. We let $\Omega = (\omega_\P, \omega_{std})$ denote the product
K\"ahler structure on $\P \times \C$ where $\omega_{std}$ is the standard
symplectic structure $\frac{i}{2} dz \wedge d\bar{z}$ on $\C$ with
respect to the usual complex coordinate $z$. Moreover, by pulling
back via the projection $\pi_1: \P \times \C \to \P$ to the first factor, we
also have a line bundle $\pi_1^*L_\P$ on $\P \times \C$; this
restricts to a line bundle $L_\X$ on the family $\X$. 
Let $\omega_t := \Omega \vert_{X_t}$ (respectively $L_t := L_\X
\vert_{X_t}$) denote the restriction of $\Omega$ (respectively $L_\X$)
to the
fiber $X_t = \pi^{-1}(t)$. 
Moreover, pulling back the prequantum data on $\P$ to $\P \times \C$
via $\pi_1$ and restricting to $\X$ yields prequantum data on $\X$.
Let $\nabla_t$ and $h_t$ denote the restrictions of $\nabla_\X$ and
$h_\X$, respectively, to the fiber $X_t$.
With this notation in place we can state
further assumptions on our toric degeneration (also see
\cite{HarKav}): 
\begin{itemize}
\item[(c)] Under the isomorphism $\varrho_1: X \to X_1$, the
  prequantum data $(\omega_1, L_1, \nabla_1, h_1)$ on $X_1$ 
  pulls back to the prequantum data $(\omega, L, \nabla, h)$ on $X$. 
\item[(d)] The K\"ahler form $\Omega$ on $\P \times \C$ is
  $T_0$-invariant, where $T_0 \cong (S^1)^n$ is the compact torus
  subgroup of the complex torus $\T_0 \cong (\C^*)^n$ acting on the
  toric variety $X_0$.
\end{itemize}

In this context, it was shown in \cite{HarKav} that $X$
admits an integrable system which is a Hamiltonian torus action on an
open dense subset of $X$. We quote the following. 

\begin{theorem} \cite[Theorem (A) in Introduction]{HarKav} \label{theorem:HarKav}
  Let $X$ be a smooth, irreducible complex algebraic variety with
$\dim_\C(X)=n$ equipped with a K\"ahler structure $\omega$. 
Let $\pi: \X \to \C$ be a toric degeneration of $X$ in the sense
described above. Suppose that $\pi: \X \to \C$ additionally satisfies
assumptions (a)-(d). Then: 
\begin{enumerate}
\item[(1)] there exists a surjective continuous map $\phi: X \to X_0$ which is a 
symplectomorphism on a dense open subset $U \subset X$ (in the classical topology),
\item[(2)] there exists a completely integrable system $\mu = (F_1, \ldots, F_n)$ on $(X, \omega)$
such that its moment map
image $\Delta$ coincides with the moment map image of $(X_0, \omega_0)$ (which is a polytope). 
\item[(3)] Let $U \subset X$ be the open dense subset of $X$ from
  (1). Then the integrable system $\mu=(F_1, \ldots, F_n)$ generates a
  Hamiltonian torus action on $U$, and the inverse image
  {$\mu^{-1}(\Delta^\circ)$ of the interior of $\Delta$ under the
    moment map $\mu: X \to \R^n$ of the integrable system}
lies in the open subset $U$.
\end{enumerate}
\end{theorem}

The main result of the present manuscript extends the above result by
additionally working with the prequantum data. 
We first state one additional assumption on the family $\X$. Since $\P \cong \p^N$ is a standard projective space, there
is a complex torus $\T_\P \cong (\C^*)^N$ acting in the standard
fashion on $\P$. By assumption (b) above, the torus $\T_0$ acting on
$X_0$ extends to a linear action on $\P$, i.e. there is an inclusion
homomorphism $\iota: \T_0 \into \T_\P$ inducing the action of $\T_0$
on $\P$, and this action preserves $X_0 \subset \P$. Similarly there
is an inclusion (by abuse of notation also denoted $\iota$) of compact subgroups $\iota:
T_0 \cong (S^1)^n \into T_\P \cong (S^1)^N$. Let $\iota^*: \t_\P^* \to
\t_0^*$ denote the corresponding dual projection.
Let $\Delta_\P \subseteq \t_\P^*$
denote the moment polytope (i.e. the moment map image) of $\P$ associated to the Hamiltonian
action on $\P$ of $T_\P$, and let $\Delta_0 \subseteq \t_0^*$ denote the 
moment polytope of $X_0$ with respect to $\T_0$. 
We will make the following genericity assumption on $X_0$: 
\begin{itemize} 
\item[(e)] the special fiber $X_0 \subseteq \P$ of our toric
  degeneration is the closure of the $\T_0$-orbit through
  $[1:1:1:\cdots:1] \in \P$.
\end{itemize} 
From this it follows by standard Hamiltonian-geometry arguments that 
$\iota^*(\Delta_\P) = \Delta_0$. 
We now define 
\begin{equation}\label{eq:def W_0}
W_0 := \iota^*(\Delta_\P \cap \Z^N) \subseteq \Delta_0 \cap \Z^n.
\end{equation}

\begin{remark} 
It is not necessarily the case that $\iota^*(\Delta_\P \cap \Z^N)  =
\Delta_0 \cap \Z^n$ even if $\iota^*(\Delta_\P) = \Delta_0$, as can be
seen from the case when $X_0$ is the closure of the image
of the embedding $\C^* \to \p^2$ given by $t \mapsto [t^2:t^3:1]$.  
\end{remark}
We call an element of $W_0$ an \emph{interior point} if it is in the interior 
if $\Delta_0$, and a \emph{boundary point} if it is on the boundary of 
$\Delta_0$.
We will show in Proposition~\ref{prop:w0-bohr-s} that $W_0$ is
the Bohr-Sommerfeld set of the integrable system defined in
Theorem~\ref{theorem:HarKav}.

We can now state the main result of this
paper. 

\begin{theorem}\label{theorem:main}
Let $X$ be a smooth irreducible complex algebraic variety with
$\dim_\C(X)=n$. Suppose $X$ is equipped with prequantum data 
$(\omega, J, L, h, \nabla)$ as above. 
Let $\pi: \X \to \C$ be a toric degeneration of $X$ 
satisfying assumptions (a)-(e). 
Let $\mu: X \to \R^n$ denote the integrable system associated to the toric degeneration 
$\X$ as in Theorem~\ref{theorem:HarKav}. We
additionally assume the following properties hold. 
\begin{enumerate}
\item[(f)] The restriction map $H^0(\P, L_\P) \to H^0(X, L)$ (where we
  identify $X \cong X_1, L \cong L_1$) is surjective. 
\item[(g)] The restriction to the respective lattices $\iota^*: (\t_\P^*)_{\Z} \to
  (\t_0^*)_{\Z}$ of the dual projection is surjective. 
\item[(h)] The dimension of the space of holomorphic sections of $L \to X$ is the cardinality
  of $W_0$, i.e. $\dim_\C(H^0(X,L)) = \lvert W_0 \rvert$. 
\end{enumerate}
Then there exists a continuous one-parameter family
$\{J_s\}_{s \in [0, \infty)}$ of complex structures on the underlying
$C^\infty$-manifold of $X$ such that the following holds. 
\begin{itemize}
\item For $s=0$, the complex structure $J_0$ agrees with the original
  complex structure on $X$.
\item For each $s \in [0, \infty)$ the triple $(X, \omega, J_s)$ is K\"ahler and the Hermitian line bundle $(\L, h, \nabla)$ induces a 
holomorphic structure $\overline{\partial}^s$ on $\L$.
\item For each $s \in [0, \infty)$ there exists a basis $\{ \sigma_s^m \mid m \in W_0\}$
of $H^0(X, \L, \overline{\partial}^s)$ such that for all interior points $m \in W_0$ the section 
$\frac{\sigma_s^m}{\abs{\sigma_s^m}_{L^1(X)}}$ converges to a delta
function supported on the Bohr-Sommerfeld fiber $\mu^{-1}(m)$ in the
following sense: there exist a covariantly constant section $\delta_m$
of $(L, X, \nabla)_{\mu^{-1}(m)}$ and a measure $d\theta_m$ on
$\mu^{-1}(m)$ such that, for any smooth section $\tau$ of the dual
line bundle $L^*$ over $X$, we have: 
\begin{equation}\label{eq:equation in main theorem}
\lim_{s \to \infty} \int_X \left\langle \tau, \frac{\sigma^m_s}{\|
    \sigma^m_s \|_{L^1(X)}} \right \rangle d(vol) =
    \int_{\mu^{-1}(m)} \langle \tau, \delta_m \rangle
  d\theta_m
\end{equation}
where $\| \cdot \|_{L^1(X)}$ denotes the $L^1$-norm with respect to the symplectic volume.
\end{itemize}

\end{theorem}

\begin{remark}\label{remark:motivation}
The essential idea behind Theorem~\ref{theorem:main} is a 
construction due to Baier, Florentino, Mourao, and Nunes \cite{BFMN}
of a varying set $\{\chi_s\}_{s \in [0,\infty)}$ (to be defined 
in Section~\ref{sec:variation}) of diffeomorphisms of the
underlying smooth manifold of the ambient projective space $\P$ which
is designed to have certain convergence properties. Specifically, let
$\{\sigma^m\}_{m \in \Delta_{\P} \cap \Z^n}$ denote the natural basis
  of $H^0(\P, L_\P)$ already mentioned above (see
  e.g. \cite{Ham-toric}) with respect to the original complex
  structure. In \cite{BFMN} the authors construct the diffeomorphisms
  $\chi_s$ precisely so that a pullback $\sigma^m_s$ of $\sigma^m$,
  defined using the $\chi_s$ at time $s$, has the form (for large
  enough $s$) of a ``bell curve'' centred at the Bohr-Sommerfeld fiber
  $\mu_{\P}^{-1}(m)$ and, as $s \to \infty$ and with appropriate
  normalizations, the bell curve gets narrower and narrower, thus
  converging to a dirac-delta distribution supported on the
  Bohr-Sommerfeld fiber. The bulk of the arguments in the present
  paper are devoted to taking this fundamental construction of
  \cite{BFMN} for the projective space $\P$ and making the necessary
  adjustments to apply them to our more general situation. We rely
  heavily on \cite{HamKon}, which already worked out some of the necessary
  steps for the case of the full flag variety. 
\end{remark}

\begin{remark}
  \begin{itemize}
  \item For $m\in W_0$ a boundary point, we are confident that 
    similar arguments will show that 
    the support of the section $\sigma^m_s$ localizes around the 
    Bohr-Sommerfeld fiber $\mu^{-1}(m)$;
    doing so in this paper, however, would require including many more 
    details of the constructions in~\cite{BFMN} and~\cite{HamKon}
    than we felt was desirable. 
    On the other hand, for the full statement of the convergence of sections,
    we do not have a sufficiently concrete topological
    description of the fiber
    $\mu^{-1}(m)$ to construct an analogue of the measure $d\theta_m$ 
    for fibres over boundary points of $W_0$.

  \item The normalization factor $\lVert \sigma^m_s \rVert_{L^1(X)}$
    in~\eqref{eq:equation in main theorem} 
   guarantees that the ``area under the bell curve'' mentioned in
   Remark~\ref{remark:motivation} is always
  equal to $1$ as $s$ varies. 

\item There are different versions of convergence in functional
  analysis, and the notion used in Theorem~\ref{theorem:main} is called ``weak
  convergence''. In particular, note that our convergence assertion is not
  uniform in the space of test sections $\tau$. 

\end{itemize} 
\end{remark}

As mentioned in the introduction, the purpose of Sections~\ref{sec:variation} and~\ref{sec:proof} is to prove 
Theorem~\ref{theorem:main}. We show that the theory of Newton-Okounkov
bodies and their associated toric degenerations provides a large class
of examples satisfying the hypotheses of Theorem~\ref{theorem:main} in
Section~\ref{sec:NOBY}.

\section{Variation of complex structures and bases of
  holomorphic sections}\label{sec:variation}

In order to prove Theorem~\ref{theorem:main} we rely on work of
Guillemin and Abreu \cite{Abreu, Guill-book, Guill-K} and, more recently, of Baier, Florentino, Mourao,
and Nunes \cite{BFMN}. Moreover, the first author and Konno
\cite{HamKon} have results similar to our Theorem~\ref{theorem:main}
for the special case of flag manifolds and its Gel'fand-Tsetlin
integrable system. In this section we recall the relevant background
and establish the preliminary results required to prove the results in
our (more general) case. 

\subsection{The gradient-Hamiltonian flow}

Let $X$ be a smooth, irreducible complex algebraic variety and
$\pi: \X \to \C$ be a toric degeneration of $X$ satisfying 
assumptions (a)-(d) as above. We equip (the smooth locus
of) $\X$ with the K\"ahler form $\omega_\X := \Omega \vert_{\X}$ as in
Section~\ref{sec-main-result}. 
The
proof of our main result will use the gradient-Hamiltonian techniques of \cite{HarKav}
which we now briefly recall.

Following Ruan \cite{Ruan}, we define the
gradient-Hamiltonian vector field corresponding to $\pi$ on
the smooth locus $\X_{smooth}$ of $\X$ as follows. Let $\nabla(\Re(\pi))$ denote the gradient
vector field on $\X_{smooth}$ associated to the real part $\Re(\pi)$, with respect to
the K\"ahler metric $\omega_\X$. Since $\omega_\X$ is K\"ahler and $\pi$ is holomorphic,
the Cauchy-Riemann equations imply that $\nabla(\Re(\pi))$ is related
to the Hamiltonian vector field 
$\xi_{\Im(\pi)}$ of the imaginary part $\Im(\pi)$ with respect to $\omega_\X$ by 
\begin{equation}\label{gradient Re pi and Hamiltonian Im pi}
\nabla(\Re(\pi)) = - \xi_{\Im(\pi)}.
\end{equation}
Let $Z$ denote the closed subset of $\X$ which is the union of the singular locus of $\X$ 
and the critical set of $\Re(\pi)$, i.e. the set on which $\nabla(\Re(\pi)) = 0$.  
The \emph{gradient-Hamiltonian vector field} $V_\pi$, which is defined only on the open set $\X \setminus Z$,
is by definition 
\begin{equation}\label{def-grad-Hamiltonian}
V_\pi := - \frac{\nabla(\Re(\pi))}{\|\nabla(\Re(\pi))\|^2}.  
\end{equation}
Where defined, $V_\pi$ is smooth.
For $t \in \R_{\geq 0}$ let $\phi_t$ denote the time-$t$ flow
corresponding to $V_\pi$. Note that since $V_\pi$ may not be complete, $\phi_t$ for a given $t$ is not
necessarily defined on all of $\X \setminus Z$; this issue is dealt
with in the next proposition. 

The gradient-Hamiltonian flow is the tool which allows us to relate
the geometry of different fibres of the toric degeneration. 
Recall that $X_t$ denotes the fiber $\pi^{-1}(t)$ and that we often
identify $X_1$ with the original variety $X$ using the isomorphism
$\varrho_1$ above. 
We now record some facts, which hold under our
assumptions, assembled from \cite[Sections 2-4]{HarKav} and which are
also used in the proof of \cite[Theorem (A)]{HarKav}.  

\begin{proposition}\label{prop-grad-Hamiltonian}
In the setting above, we have the following. 
\begin{itemize}
\item[(a)] Let $s,t \in \R$ with $s \geq t > 0$. Where defined, the flow $\phi_t$ takes $X_s \cap (\X \setminus Z)$ to
$X_{s-t}$. In particular, where defined, $\phi_t$ takes a point $x \in
X_t$ to a point in the fiber $X_0$. Moreover, 
for $s>t>0$, the map $\phi_t$ is defined on all of $X_s$ and it is a
diffeomorphism from $X_s$ to $X_{s-t}$. 
\item[(b)] Where defined, the flow $\phi_t$ preserves the symplectic
  structures, i.e., if $x \in X_z \cap (\X \setminus Z)$ is a point
  where $\phi_t(x)$ is defined, then
  $\phi_t^*(\omega_{z-t})_{\phi_t(x)} = (\omega_z)_x$. In particular,
  for $s>t>0$, the map $\phi_t$ is a symplectomorphism between $X_s$
  and $X_{s-t}$. 
\item[(c)] For $s=t$, there exists an open dense subset $U_t=U_s$ of
  $X_t$ and an open dense subset $U_0 \subset X_0$ in the smooth locus
  of $X_0$ such that $\phi_t$ is a symplectomorphism from $U_t$ to
  $U_0$. Moreover, $\phi_t$ extends continuously to a map $\phi_t: X_t
  \to X_0$. 
\end{itemize}
\end{proposition}

Using the gradient-Hamiltonian flows, for each $0<t\leq 1$ 
we construct an integrable $\mu_t \colon X_t \to \R^n$ 
by pulling back the standard integrable system 
$\mu_0: X_0 \to \R^n$ 
(arising from the structure of $X_0$ as a toric variety) 
on $X_0$ through the maps $\phi_t$ \cite[Theorem
5.2]{HarKav}.  
More precisely, we define 
\begin{equation}\label{eq:definition mu_t}
\mu_t := \phi_t^* \mu_0: X_t \to \R^n.
\end{equation}
As a result, 
the moment map image $\mu_t(X_t)$ for each $0 < t \leq 1$ is equal to the
moment map image $\Delta_0 := \mu_0(X_0)$ of the toric variety $X_0$.
In particular, we have an integrable system 
$\mu: X \to \R^n$ on $X\cong X_1$
whose image is $\Delta_0$.  For further details we refer to
\cite{HarKav}. 

In what follows it will sometimes be useful to refer to the
families of $X_t$ and $U_t$ with $t\in [0,1]$, and so we define
\begin{equation}\label{eq:X-and-U-fam}
  \begin{split}
\X_{[0,1]} &:= \pi\inv\bigl([0,1]\bigr)\subset \X \\
\mathcal{U}_{[0,1]} & := \{ x \in \X_{[0,1]} \st x \in U_{\pi(x)}\}.
  \end{split}
\end{equation}

\subsection{The varying complex structure}\label{ss:varying-cplx-str}

We now define a family $\{J_{s,t}\}$ of complex structures on $X \cong
X_1$ where $s, t$ are real parameters with $0 \leq s < \infty$ and $0
< t \leq 1$. In Section~\ref{sec:proof} 
we will choose an appropriate continuous function $t = t(s)$ of $s$
and thus define a $1$-parameter family $\{J_s = J_{s,t(s)}\}$ of
complex structures which will satisfy the properties asserted in our
Theorem~\ref{theorem:main}. 

For details we refer to \cite{BFMN, HamKon} but we
briefly set some notation.
Recall that $\P \cong \p^N$ is a
standard projective space. By slight abuse of notation we denote also
by $\P$ the underlying smooth manifold.
Since the usual
projective space $\p^N$ is naturally a K\"ahler manifold with K\"ahler
structure $(\P, \omega_\P, J_\P)$, 
we may consider $\P$ as a symplectic manifold
$(\P, \omega_\P)$ 
or as a complex manifold $(\P, J_\P)$. 
In \cite{BFMN} the authors construct a family of diffeomorphisms
\begin{equation}\label{eq:chi_s}
\chi_s: \P \to \P
\end{equation} 
for $s \in \R$ with 
$0 \leq s < \infty$ which satisfy the following:
\begin{itemize}
\item[($\chi$-1)] $\chi_0: \P \to \P$ is the identity function, and

\item[($\chi$-2)] for any $s$ with $0 \leq s < \infty$, the triple
  $(\P, \omega_\P, \chi_s^*(J_\P))$ is a K\"ahler structure on $\P$.
\end{itemize} 
The family $\{\chi_s\}$ in~\eqref{eq:chi_s} satisying
($\chi$-1) and ($\chi$-2) is not
uniquely determined; the general construction given in \cite{BFMN} could yield
many such choices of $\{\chi_s\}$. 
We interpret the diffeomorphisms $\chi_s$ as giving rise to a one-parameter family of K\"ahler structures on $\P$ with respect to the same symplectic structure but with varying complex structure.

In this paper, we wish to use the varying complex structures
$\chi_s^* J_\P$ on $\P$ to define a family of complex structures
on $X$. However, this is not completely
straightforward because a smooth submanifold 
$X$ of $\P$
may be a complex submanifold of $\P$ for the original
complex structure $J_\P$ 
but may \emph{not} be a complex submanifold of
$\P$ equipped with the altered complex structure $\chi_s^*J_\P$. 
To address this issue, the first author and Konno
prove the following \cite[Proposition 6.1]{HamKon}. 
 
\begin{proposition}\label{prop:embedding of submanifold}
(\cite[Proposition 6.1]{HamKon}) 
Let $V$ be a smooth submanifold of $\P$ with associated embedding $\rho: V
\into \P$. Assume that $V$ is a complex 
submanifold of $\P$ with respect to the complex
  structure $J_\P$ and let 
  $\omega_V := \rho^*(\omega_\P)$ denote the corresponding K\"ahler
  form on $V$. Let $\chi_s: \P \to \P$ for $s$ a real parameter, $0 \leq s <
  \infty$, be a family of diffeomorphisms as in~\eqref{eq:chi_s}
  satisfying ($\chi$-1) and ($\chi$-2). Then 
 there exists a family $\{\rho_s\}_{0 \leq s < \infty}$ of embeddings
 $\rho_s: V \into \P$ such that 
\begin{enumerate} 
\item[(a)] for all $s$ with $0 \leq s < \infty$ we have
  $\rho_s^*\omega_\P \vert_V = \omega_V$, 
\item[(b)] for all $s$ with $0 \leq s < \infty$ the image $\rho_s(V)
  \subset \P$ is a complex submanifold of $(\P, \chi_s^*(J_\P))$, and 
\item[(c)] $\rho_0 = \rho$. 
\end{enumerate}
In particular, for any $s$ with $0 \leq s < \infty$ the pair $\left(\omega_V, \rho_s^*\left(\chi_s^*(J_\P) \vert_{\rho_s(V)}\right)\right)$ is a K\"ahler structure on $V$. 
Furthermore, for each choice of family $\{\chi_s\}$ as
in~\eqref{eq:chi_s}, the family of embeddings $\{\rho_s\}$ satisfying the
above conditions is unique. 
 
\end{proposition}

Each fiber $X_t$ (for $0 < t \leq 1$) of our family $\X$ is a complex submanifold of $\P$
with respect to the original complex structure $J_\P$.  Hence, applying
Proposition~\ref{prop:embedding of submanifold} to each $X_t$, we obtain embeddings 
\begin{equation}\label{def:rho s t} 
\rho_{s,t}: X_t \to \P
\end{equation}
where $s, t$ are real parameters with $0 \leq s < \infty$ and $0 <
t \leq 1$. We can now define a family
$\{J_{s,t}\}$ of complex structures on $X=X_1$. 
We have the following key diagram (note that it is not commutative): 
\begin{equation}\label{eq:key diagram} 
\xymatrix{ 
& \rho_{s,t}(X_t) \ar[r]^{\subset} & (\P,\omega_\P) \ar[r]^{\chi_s} & (\P,J_\P) \\
X=X_1 \ar[r]^{\phi_{1-t}} & X_t \ar[u]^{\rho_{s,t}} \ar[r]^{\phi_t} & X_0 \ar[u]  & \\
}
\end{equation}
and we have the following. 

\begin{definition}
Let $s,t \in \R$ with $0 \leq s < \infty$ and $0 < t \leq
1$. Let $\phi_{1-t}: X \cong X_1 \to X_t$ 
be the gradient-Hamiltonian flow and let $\{\chi_s\}$ be a choice of
diffeomorphisms as in~\eqref{eq:chi_s} and
$\rho_{s,t}: X_t \into \P$ 
be the corresponding embeddings in~\eqref{def:rho s t}.
The complex structure $J_{s,t}$ on $X$ is then defined by 
\begin{equation}\label{eq:definition Jst}
J_{s,t} ;= (\rho_{s,t} \circ \phi_{1-t})^*(\chi_s^* J_\P
\vert_{\rho_{s,t}(X_t)})
\end{equation}
Equivalently, $J_{s,t}$ is the pullback
$(\chi_s \circ \rho_{s,t} \circ \phi_{1-t})^* J_\P$. 
\end{definition}

Since both $\rho_{s,t}$ and $\phi_{1-t}$ behave 
well with respect to the symplectic structures
(Propositions~\ref{prop-grad-Hamiltonian}(b) and~\ref{prop:embedding
  of submanifold}(a)), the
following is immediate.

\begin{lemma}\label{lemma:pairs Kahler}
Let $s, t \in \R$ with $0 \leq s < \infty$ and $0 < t \leq 1$. Then 
the triple $(X\cong X_1, \omega = \omega_1, J_{s,t})$
  is K\"ahler. 
\end{lemma}

In what follows, we will need two further properties of
these embeddings $\rho_{s,t}$ and complex structures $J_{s,t}$, the
first of which requires some additional hypotheses on the family $\{\chi_s\}$, as
we now explain. 
As mentioned above, the family $\{\chi_s\}$ given in~\eqref{eq:chi_s}
is not unique. However, when $\P$ is equipped with a complex torus action and
$V$ happens to be the closure in $\P$ of a torus orbit, then it turns out
that the family $\{\chi_s\}$ can be chosen in such a way that $V$
remains a complex submanifold for \emph{all} of the complex structures
$\chi_s^* J_\P$, not just the original complex structure $J_\P$. 
Before proceeding it should be noted that although the statement of
Proposition~\ref{prop:embedding of submanifold} (equivalently
\cite[Proposition 6.5]{HamKon}) contains the hypothesis that $V$ is
smooth, it is shown in the proof of \cite[Proposition 6.5]{HamKon}
that the argument for \cite[Proposition 6.1]{HamKon} can be extended
in this special case to give a well-defined embedding $\rho$ of $V$, with
analogous properties.

\begin{proposition}\label{prop:toric submanifolds don't move}
Let $\H \subset \T_\P$ be a complex subtorus, acting on $\P$ by
restriction of the standard $\T_\P$-action on $\P$. Let 
$V$ denote the (possibly singular) closure in $\P$ of the $\H$-orbit
of $[1:1:\cdots:1] \in \P$.  
Then there exists a choice of a family $\{\chi_s\}$ as
in~\eqref{eq:chi_s}, satisfying the assumptions ($\chi$-1) and ($\chi$-2),
such that 
for all $s$ with $0 \leq s < \infty$ we have $\rho_s = \rho_0$, where
$\rho_s$ is the (unique) embedding associated to $\chi_s$ constructed
in Proposition~\ref{prop:embedding of submanifold} above. In
particular, on the smooth locus of $V$, the complex structures
$\chi_s^* J_\P$ and $J_\P$ agree, for all $s$. 
\end{proposition}

Since the technical aspects of the
construction of the above family $\{\chi_s\}$ 
are not used in this manuscript, we do not discuss it
further here; for details see \cite{HamKon, BFMN}. 
In the setting of this manuscript, the fiber $X_0 :=
\pi^{-1}(0)$ over $0$ of our family $\X$ is by assumption the closure
of the $\T_0$-orbit of $[1:1:\cdots:1]$ for $\T_0 \subset \T_\P$ a
subtorus of $\T_\P$. Hence Proposition~\ref{prop:toric
  submanifolds don't move} applies, and we therefore obtain a family
$\{\chi_s\}$ of diffeomorphisms as in~\eqref{eq:chi_s} (satisfying the
assumptions ($\chi$-1) and ($\chi$-2)) such that the associated
$\rho_s$'s leave $X_0$ invariant,
i.e., $\rho_s=\rho_0$ on $X_0$ for all $0 \leq s < \infty$. This will
be crucial in what follows so we now record, by way of emphasis, that 
\begin{quote}
\textbf{henceforth, we assume that the family $\{\chi_s\}$ is chosen
  in such a way that the conclusion of Proposition~\ref{prop:toric
    submanifolds don't move} holds.} 
\end{quote}

Given this choice of $\{\chi_s\}$, 
in our later arguments we need to know that the maps $\rho_{s,t}$
defined in~\eqref{def:rho s t} 
satisfy some continuity conditions with respect to the parameter $t$.
We record the following.

\begin{proposition}\label{prop:rho-smooth-in-t}
  Let $s \in \R$ with $0 \leq s < \infty$. Let
  $\X \subseteq \P \times \C$ be the toric degeneration as above
  and $\{\chi_s\}$ be a family of diffeomorphisms as above, 
chosen so that the conclusion of
  Proposition~\ref{prop:toric submanifolds don't move}
  holds. Then 
\begin{enumerate}
\item the map $\X_{[0,1]} \to \P$ given by $x \mapsto
  \rho_{s,\pi(x)}(x)$ is continuous, and  
\item the map $\U_{[0,1]} \to \P \times \C$ given by $x \mapsto
  (\rho_{s,\pi(x)}(x), \pi(x))$ is a diffeomorphism onto its image. 

\end{enumerate}
\end{proposition}

\begin{proof}

We first review the construction of the map $\rho_s$ from~\cite{HamKon}.  
Let $\chi_s\colon \P \to \P$ be the diffeomorphisms from~\eqref{eq:chi_s}.  
Let $\psi_s = \chi_0 \circ \chi_s\inv \colon \P \to \P$ 
and let $\omega_s = (\chi_s\inv)^* \omega_\P$. By property ($\chi$-2)
of the family $\{\chi_s\}$ we know that $(\omega_s, J_\P)$ is a
K\"ahler structure on $\P$. Thus, any submanifold $V$ of $\P$
which is complex with respect to $J_\P$ is also a symplectic
submanifold with respect to $\omega_s$. 
Hence we can 
define a time-dependent vector field 
$\mathbb{V}_s$ on $\P$ by 
\[
(\mathbb{V}_s)_{\psi_s(p)} = \frac{d}{d\tau} \psi_{s+\tau} (p) \Bigr\rvert_{\tau=0} 
\]
for $p\in \P$.  
Following~\cite{HamKon} we further define a vector field $Y_s$ on $V$ by 
\begin{equation}\label{eq:Y_s}
\iota_{Y_s}\bigl(\omega_s\vert_V\bigr) 
= -\psi_s^* \bigl(\iota_{\mathbb{V}_s}\omega_\P\bigr)\bigr\rvert_V.
\end{equation}
Letting $\varphi_s$ denote the corresponding flow of the vector field
$Y_s$, again following \cite{HamKon} we finally define $\rho_s$ by 
\[
\rho_s := \chi_s^{-1} \circ \rho_0 \circ \varphi_s \circ \chi_0
\bigr\rvert_{V}. 
\]
The construction just recounted deals with a single submanifold $V
\subseteq \P$. To prove the proposition we must show that this
construction can be extended to one on a family $\X \subseteq \P
\times \C$ in a way which guarantees the claimed smoothness and
continuity properties with respect to the extra parameter. To
do this, we first define $\widehat{\chi}_s: \P \times \C \to \P \times \C$
by $\widehat{\chi}_s(x,t) = (\chi_s(x), t)$ for $(x,t) \in \P \times \C$
and $\widehat{\psi}_s := \widehat{\chi}_0 \circ \widehat{\chi}_s^{-1}$. We then
define a time-dependent vector field $\widehat{\mathbb{V}}_s$ on $\P
\times \C$ by 
\[
(\widehat{\mathbb{V}}_s)_{(\psi_s(p),t)} :=  \frac{d}{d\tau}
\widehat{\psi}_{s+\tau}(p,t) \Bigr\rvert_{\tau=0}
\]
and a vector field $\widehat{\mathbb{Y}}_s$ on $\X_{smooth}$, the smooth
locus of $\X$, by 
\begin{equation}\label{eq:def hat Y_s}
\iota_{\widehat{\mathbb{Y}}_s}(\widehat{\omega}_s) =
\iota_{\widehat{\mathbb{Y}}_s}((\widehat{\chi}_s^{-1})^* \Omega) = -
\widehat{\psi}_s^*(\iota_{\widehat{\mathbb{V}}_s} \Omega) \bigr\vert_{\X_{smooth}}
\end{equation}
where $\Omega$ is the product K\"ahler structure on $\mathcal{P} \times \C$. 
Note that $\widehat{\mathbb{V}}_s = (\mathbb{V}_s,0)$ by definition of
$\widehat{\psi}_s$ and by construction $\widehat{\mathbb{Y}}_s$ is a smooth
vector field on $\X_{smooth}$. We
wish to analyze the relation between $Y_s$, defined via the above
construction from \cite{HamKon} on each $X_t$ separately, and
$\widehat{\mathbb{Y}}_s$, for which we need some preliminaries. Let $pr_1:
\P \times \C \to \P$ be the projection to the first factor and
$\mathcal{V} \subseteq T\X_{smooth}$ denote the vertical subbundle of
$T\X_{smooth}$ with respect to $pr_1$, i.e., $\mathcal{V}_x :=
\ker(d(pr_1)_x)$ for each $x \in \X_{smooth}$. Then $\mathcal{V}$ is a
smooth symplectic subbundle of $T\X_{smooth}$ with respect to $\Omega
\vert_{\X_{smooth}}$, so there is a canonical decomposition
$T\X_{smooth} \cong \mathcal{V} \oplus \mathcal{V}^{\Omega}$ and the
projection $T\X_{smooth} \to \mathcal{V}$ is smooth. 

We now claim that 
\begin{equation}
  \label{eq:form is a pullback from base}
  - \widehat{\psi}_s^*(\iota_{\widehat{\mathbb{V}}_s} \Omega) = pr_1^* \big( -
  \Psi_s^*(\iota_{\mathbb{V}_s} \omega_{\P}) \big)
\end{equation}
as $1$-forms on $\P \times \C$. Indeed, for any $w \in T(\P \times \C)
\cong T\P \oplus T\C$, we may decompose $w=(w_\P, w_\C)$ into its two
factors and compute 
\begin{equation*}
  \begin{split}
\widehat{\psi}_s^*(\iota_{\widehat{\mathbb{V}}_s} \Omega)(w) & =
\Omega(\widehat{\mathbb{V}}_s, (\widehat{\psi}_s)_*(w)) \\
 & = \omega_{\P}(\mathbb{V}_s, ((\widehat{\psi}_s)_*w)_{\P})  \textup{
   since } \widehat{\mathbb{V}}_s = (\mathbb{V}_s, 0) \\
 & = \omega_{\P}(\mathbb{V}_s, (\Psi_s)_*(\omega_{\P})) \textup{ since
   $\widehat{\psi}_s$ acts as the identity on the $\C$ factor} \\
 & = (\Psi_s)_*(\iota_{\mathbb{V}_s} \omega_{\P})(w_{\P}) \\
 & = pr_1^*(\Psi_s^*(\iota_{\mathbb{V}_s} \omega_{\P}))(w).    
  \end{split}
\end{equation*}
Now suppose $Z \in \mathcal{V} \subseteq T\X_{smooth}$, so $Z = (Z_\P,
0)$ where $Z_{\P} \in T\P  \cap T\X_{smooth}$. We have 
\begin{equation*}
  \begin{split}
    \iota_{\widehat{\mathbb{Y}}_s}(\widehat{\omega}_s)(Z) & = -
    \widehat{\psi}_s^*(\iota_{\widehat{\mathbb{V}}_s} \Omega)(Z) \\ 
   & = pr_1^*(- \Psi_s^*(\iota_{\mathbb{V}_s} \omega_{\P}))(Z) \\
 & = - \Psi_s^*(\iota_{\mathbb{V}_s} \omega_{\P})(Z_\P) \\
 & = \iota_{Y_s}(\omega_s)(Z_\P)
  \end{split}
\end{equation*}
where the first equality is by~\eqref{eq:def hat Y_s}, the second
by~\eqref{eq:form is a pullback from base},
and the last is the definition~\eqref{eq:Y_s} of $Y_s$ on each fiber. From
this it follows that the symplectic-orthogonal projection
$(\widehat{Y}_s)_{vert}$ of $\widehat{Y}_s$ to the vertical subbundle
$\mathcal{V} \subseteq T\X_{smooth}$ agrees, fiberwise, with the
vector field $Y_s$ defined using the original construction from
\cite{HamKon}. Since the symplectic-orthogonal projection is smooth,
as argued above, it follows that the vector field $Y_s$, considered
together on all of $\X_{smooth}$, is smooth on $\X_{smooth}$. Let
$\widehat{\varphi}_s$ denote the flow corresponding to
$(\widehat{Y}_s)_{vert}$, which exists since it exists fiberwise for each
$X_t$ with $t\neq 0$, and for $X_0$, the argument in the proof of
\cite[Proposition 6.1]{HamKon} shows that a flow exists and extends
continuously to all of $X_0$. 
Hence the statement (1) of
the proposition now follows. 

To prove (2), we first observe that the above argument shows that the
map $\mathcal{U}_{[0,1]} \to \P \times \C$ given by $x \mapsto
(\rho_{s, \pi(x)}(x), \pi(x))$ is smooth. So it suffices to show that
this map is smoothly invertible. We know that for each fixed $t \in
\C$ with $t \neq 0$, the map $\rho_{s,t}: X_t \to \P$ is an embedding. Moreover,
Proposition~\ref{prop:toric submanifolds don't move} implies that on $U_0 \subseteq X_0$ we have
$\rho_{s,0}=\rho_0$, hence $\rho_{s,0}$ is also an embedding. It
follows that $x \mapsto (\rho_{s,\pi(x)}(x), \pi(x))$ is injective, so
it is bijective on its image. It remains to show that the inverse map
is also smooth. Since $\mathcal{U}_{[0,1]}$ lies in $\X_{smooth}$ we
may decompose $T\mathcal{U}_{[0,1]}$ into the vertical subbundle
$\mathcal{V}$ and its complement $\mathcal{V}^{\Omega}$. With respect
to this decomposition and the standard decomposition $T\P \oplus T\C$
of $\P \times \C$, the derivative of the above map at a point in
$\pi^{-1}(t)$ is of the form 
\[
\begin{bmatrix} (\rho_{s,t})_* & \star \\ 0 & I \end{bmatrix} 
\]
where $I$ is an isomorphism and $(\rho_{s,t})_*$ is injective. Thus
the whole derivative is also injective, and it follows that the
inverse mapping is smooth.  
\end{proof}

\subsection{Pullbacks of prequantum data}\label{subsec:lifts}

The main result of this manuscript deals with quantizations, and in
particular with sections of certain prequantum line bundles. 
In this section, we show that the gradient-Hamiltonian flows
$\phi_{1-t}$ and the embeddings $\rho_{s,t}$ from
Section~\ref{ss:varying-cplx-str} lift to the total spaces of the
relevant line bundles. This will be crucial for our constructions
below. Recall that we have the prequantum data
$(L_\P, \nabla_\P, h_\P)$ and $(L_\X, \nabla_\X, h_\X)$ respectively
on $\P$ and $\X$ and that the latter restricts to give prequantum data
on the fibers $X_t$, which we denote by $(L_t, \nabla_t, h_t)$.

We first recall that the horizontal lift of the gradient-Hamiltonian
flow with respect to the connection $\nabla_\X$ 
preserves the connections and Hermitian metrics on
each fiber \cite[Proposition 4.3]{HamKon}. 

\begin{lemma}\label{lemma:lift-grH}(\cite[Proposition 4.3]{HamKon}) 

Let $s, t \in \R$ with $s\geq t > 0$. 
\begin{enumerate} 

\item If $s>t$, there exists a unique horizontal lift $\tilde{\phi}_t:
  L_s \to L_{s-t}$ of the gradient-Hamiltonian flow $\phi_t: X_s \to
  X_{s-t}$ to the total spaces of the prequantum line bundles. The
  lift $\tilde{\phi}_t$ is an isomorphism of line bundles and also preserves the fiberwise connections and
Hermitian structures, i.e., $\tilde{\phi}_t^* \nabla_{s-t} = \nabla_s$ 
and $\tilde{\phi}_t^* h_{s-t} = h_s$.

\item If $s=t$, there there exists a unique horizontal lift $\tilde{\phi}_{s=t}:
  L_t \vert_{U_t} \to L_0 \vert_{U_0}$ of the gradient-Hamiltonian flow $\phi_t: U_t \to
  U_0$ to the total spaces of the (restricted) prequantum line bundles. The
  lift $\tilde{\phi}_t$ is an isomorphism of line bundles and also preserves the fiberwise connections and
Hermitian structures, i.e., $\tilde{\phi}_t^* \nabla_{0} = \nabla_t$ 
and $\tilde{\phi}_t^* h_{0} = h_t$ (restricted to $U_t$
and $U_0$). 

\item 
The map
$L_0 \vert_{U_0} \times [0,1] \to L_\X$ given by $(x,t) \mapsto
\tilde{\phi}_t^{-1}(x)$ is smooth, where the domain $L_0 \vert_{U_0}
\times [0,1]$ is given the standard smooth structure induced from the
product structure. 
\end{enumerate}

\end{lemma} 

\begin{proof}
  The argument is essentially the same as in \cite{HamKon}. For the
  statement in (3) we note that the gradient-Hamiltonian flow on $\X
  \setminus Z$ is smooth and hence its horizontal lift is also
  smooth. Now an argument similar to that in the proof of
  Proposition~\ref{prop:rho-smooth-in-t} yields the result. 
\end{proof}

We remark that it follows from the above lemma that the following
diagram commutes for $t$ with $0 < t < 1$:
\[
\xymatrix{ 
L_1 \ar[d] \ar[r]^{\tilde{\phi}_{1-t}} & L_{t} \ar[d] \\
X\cong X_1 \ar[r]^{\phi_{1-t}} & X_{t} \\
}
\]

We will also need a similar statement for the embeddings
$\rho_{s,t}$ \cite[Proposition6.3(1)]{HamKon}. 

\begin{lemma}\label{lemma:lift-rhos} (\cite[Proposition
  6.3(1)]{HamKon}) 
There exists a lift $\tilde{\rho}_{s,t}: L_t \to L_\P
\vert_{\rho_{s,t}(X_t)}$ of $\rho_{s,t}$ to
the total spaces of the line bundles $L_t$ and $L_\P
\vert_{\rho_{s,t}(X_t)}$ which identifies the prequantum
data.  
In particular, the following diagram commutes:
\[
\xymatrix{ 
L_t \ar[d] \ar[r]^{\tilde{\rho}_{s,t}} & L_\P \vert_{\rho_{s,t}(X_t)} \ar[d] \\
X_t \ar[r]^{\rho_{s,t}} & \rho_{s,t}(X_t) \\
}
\]
\end{lemma} 

We need a smoothness property for the map $\widetilde{\rho}_{s,t}$,
analogous to Proposition~\ref{prop:rho-smooth-in-t} for the $\rho_{s,t}$. 

\begin{proposition}\label{prop:rho-tilde-smooth-in-t}
  Let $s \in \R$, $0 \leq s < \infty$ be fixed. Let
  $\X \subseteq \P \times \C$ be the toric degeneration as above
  and $\{\chi_s\}$ be the family of diffeomorphisms as above (in
  particular chosen so that the conclusion of
  Proposition~\ref{prop:toric submanifolds don't move}
  holds). For any $0 \leq t \leq 1$, let $\rho_{s,t}:
  X_t \into \P$ be the embedding defined in
  Proposition~\ref{prop:embedding of submanifold}, and 
let $\tilde{\rho}_{s,t}$ be the lifting of $\rho_{s,t}$ as defined in 
Lemma~\ref{lemma:lift-rhos}. Let $L_\U$ denote the restriction of the
line bundle $L_\X$ to the open subset $\U_{[0,1]}$ defined in~\eqref{eq:X-and-U-fam}. 
Then the map
\begin{equation}\label{eq:formula for tilderhos}
\tilde{\rho}_s: L_\U \to L_\P \times \C, \quad 
(x,\xi) \mapsto (\tilde{\rho}_{s,\pi(x)}(\xi), \pi(x))
\end{equation} 
where the pair $(x,\xi)$ consists of a point $x \in \U$ and $\xi \in
L_x$, is a diffeomorphism onto its image. 
\end{proposition}

\begin{proof}
  We recall the construction of $\tilde{\rho}_{s,t}$ from~\cite{HamKon}
  and globalize it to the family $\X_{smooth}$.
Let $\Theta\colon \X_{smooth}\cross [0,\infty) \to \P\cross\C$ be the map
$(x,s)\mapsto \bigl( \rho_{s,\pi(x)}(x), \pi(x)\bigr)$.
Let $(L',\nabla',h')=\Theta^*(L_{\P\cross\C}, \nabla_{\P\cross\C}, h_{\P\cross\C})$.
Note that $L'\vert_{\X_{smooth}\cross\{s_0\}} = \rho_{s_0}^* L_{\X}$.

Let $\mathcal{Z}\in \operatorname{Vect}(L')$ be the horizontal
lift to $L'$ with respect to $\nabla'$ of the vector
field $\dd{}{s}$ on $\X_{smooth}\cross [0,\infty)$. 
Then the flow of $\mathcal{Z}$ through time $s$ induces a
diffomorphism between the bundles 
$\rho_0^* L_{\X}$ and $\rho_{s}^* L_{\X}$ over 
$\X_{smooth}\cross\{0\}$ and $\X_{smooth}\cross\{s\}$. 
This is the
same as a diffeomorphism $L_{\X} \to L_{\rho_{s_0}(\X)}$ lifting the
map $\rho_{s_0}$, and we denote it by $\tilde{\rho}_s$.
Since $\mathcal{U}_{[0,1]}$ is a subset of $\X_{smooth}$ by construction, the map $\tilde{\rho}_s$
is defined on $L_{\mathcal{U}}$,
and is a diffeomorphism onto its image.

If we restrict the map $\Theta$ to the (smooth locus in the)
fibre $X_t\cross [0,\infty)$ we obtain
a map $\Theta_t \colon (x,s) \mapsto \rho_{s,t}(x)$,
which agrees with the map used in~\cite[Claim 6.4]{HamKon} 
to construct the lift of $\rho_{s,t}$ over the submanifold $X_t$.
Furthermore, it is clear 
that $\Theta_t^*(L_{\P\cross\C}, \nabla_{\P\cross\C}, h_{\P\cross\C})$
restricts to $(L',\nabla',h')$ on $X_t\cross [0,\infty)$; 
this agrees with the data used in the construction in~\cite{HamKon}.
Therefore the lifting $\tilde{\rho_s}$ constructed above
agrees on $U_t$ with the map $\tilde{\rho}_{s,t}$
as constructed in~\cite{HamKon},
and the formula given in~\eqref{eq:formula for tilderhos} agrees with the map
$\tilde{\rho}_s$ constructed in the previous paragraph.

\end{proof}

Finally, we analyze the behavior of the diffeomorphisms $\chi_s: \P
\to \P$ of~\eqref{eq:chi_s} with respect to the prequantum data. 
Recall that $L_\P$ is a holomorphic line bundle with respect to
the canonical complex structure $J_\P$ on $\P$ (i.e. its transition
functions are holomorphic), and hence there exists a differential operator
$\overline{\partial}$ defining the space of 
holomorphic sections $H^0(\P, L_\P, \overline{\partial})$ of
$L_\P$ over $(\P, J_\P)$. 
We recall the following \cite[Theorem 5.3(A)]{HamKon}. 

\begin{lemma} (\cite[Theorem 5.3(A)]{HamKon}) \label{lemma:chi-tilde}
There exists a lift $\tilde{\chi}_s$ of $\chi_s$ to
an isomorphism of the line bundle $L_\P$
such that the following diagram commutes 
\[
\xymatrix{ 
L_\P \ar[d] \ar[r]^{\tilde{\chi}_{s}} & L_\P \ar[d] \\
\P \ar[r]^{\chi_{s}} & \P \\
}
\]
and such that the connection $\nabla_\P$ is the canonical Chern connection for
  the Hermitian holomorphic line bundle $(L_\P, h_\P,
  \tilde{\chi}_s^*\overline{\partial})$. 
\end{lemma}

From now on we notate 
\[
\overline{\partial}_s := \tilde{\chi}_s^*(\overline{\partial}) 
\]
and denote by $H^0(\P, L_\P, \overline{\partial}_s)$ the corresponding
space of sections.
From the definitions of the respective holomorphic structures, it is
immediate that the pullback by $\tilde{\chi}_s^*$ of a section which
is holomorphic with respect to $\overline{\partial}$ is holomorphic
with respect to $\overline{\partial}_s$. 

\begin{lemma} 
The pullback $\tilde{\chi}_s^*(\sigma)$ of a
 section $\sigma \in H^0(\P, L_\P, \overline{\partial})$
 is an element of $H^0(\P, L_\P,
 \overline{\partial}_s)$.
\end{lemma}

In fact, in our arguments below we will need to pull back sections to
the original variety $X$ via the diagram 
\begin{equation}\label{eq:big diagram for pulling back sections}
\xymatrix{ 
L \ar[r]^{\tilde{\phi}_{1-t}} \ar[d] & L_t \ar[r]^{\tilde{\rho}_{s,t}}
  \ar[d] & L_\P \ar[d] \ar[r]^{\tilde{\chi}_{s}} & L_\P \ar[d] \\
X \ar[r]^{\phi_{1-t}} & X_t \ar[r]^{\rho_{s,t}} & \P \ar[r]^{\chi_{s}} & \P \\
}
\end{equation}
obtained by composing the three diagrams above. We record the
following. 

\begin{lemma}
  Let $s, t \in \R$ with $0 \leq s< \infty$ and $0
  \leq t < 1$. Following notation as above, the pullbacks 
  $\tilde{\chi}^*_s$, $\tilde{\rho}_{s,t}^*$ and
  $\tilde{\phi}^*_{1-t}$ preserve holomorphic sections, so in
  particular there is a well-defined map 
\[
\tilde{\phi}^*_{1-t} \circ \tilde{\rho}_{s,t}^* \circ
\tilde{\chi}^*_s: H^0(\P, L_\P, \overline{\partial}_0) \to H^0(X, L,
\overline{\partial}_{s,t})
\]
where $\overline{\partial}_0$ denotes the
standard holomorphic structure on $(\P=\p^N, \omega_\P, J_\P)$. 
\end{lemma}

\begin{proof} 
The differential operator 
$\overline{\partial}_{s,t}$ is associated to the complex structure 
$\tilde{\phi}_{1-t}^* \tilde{\rho}_{s,t}^* \tilde{\chi}_s^*(J_\P)$
obtained by pullback, so the result is immediate from the definitions. 
\end{proof}

We also record the important fact that the Bohr-Sommerfeld fibres
of the integrable system on $X$ correspond to those on $X_0$.
In particular, the Bohr-Sommerfeld set is the set $W_0$
defined in~\eqref{eq:def W_0}.

\begin{proposition}\label{prop:w0-bohr-s}
  Let $W_0 := \iota^*(\Delta_\P \cap \Z^N) \subseteq \Delta_0 \cap \Z^n$
  as defined in~\eqref{eq:def W_0}.
  Then the Bohr-Sommerfeld fibres in $X$ are precisely the preimages of
  points in $W_0$ under the integrable system $\mu$ constructed
  in~\eqref{eq:definition mu_t}.
\end{proposition}

\begin{proof}
Since the gradient-Hamiltonian flow lifts to the line bundle
preserving the connection, 
it follows that the gradient-Hamiltonian flow maps Bohr-Sommerfeld
fibres in $X$ to Bohr-Sommerfeld fibres in $X_0$.
Since the Bohr-Sommerfeld fibres of the torus moment
map are the preimages of $W_0$, and
the integrable system $\mu\colon X \to \R^n$ was constructed by pulling back
the moment map for the torus action on $X_0$,
we obtain the result.
\end{proof}

\subsection{Varying bases of sections $\{\sigma^m_{s,t}\}$ of
  $H^0(X, L, \overline{\partial}_{s,t})$}\label{subsec:varying bases}

Our main result, Theorem~\ref{theorem:main}, asserts the existence of
a basis of sections $\{\sigma^m_s\}$, indexed by $m \in W_0$ and
dependent on a real parameter $s$, where each $\sigma^m_s$ is
holomorphic with respect to the complex structure $J_s$. In this
section, we take a step in this direction by using the results of Section~\ref{subsec:lifts} to
define sections $\sigma^m_{s,t}$ in 
$H^0(X,L,\overline{\partial}_{s,t})$.

The sections $\sigma^m_{s,t}$ are constructed using certain standard
sections of $(\P, L_\P)$ which we now recall. It is well-known that
$\P$ is a toric variety with respect to the standard action of its
torus $T_\P$. For any integer lattice point in $\Delta_\P \cap \Z^N$,
there is a well-known method (see e.g. \cite{Ham-toric})
to associate to it a holomorphic section in
$H^0(\P, L_\P, \overline{\partial}_0)$.  
In fact, this association
yields a bijective correspondence between $\Delta_\P \cap \Z^N$ and a
basis for $H^0(\P, L_\P, \overline{\partial}_0)$ which we denote as 
\[
\tilde{m} \in \Delta_\P \cap \Z^N \mapsto \sigma^{\tilde{m}} \in
H^0(\P, L_\P, \overline{\partial}_0).
\]
Now recall that $W_0 := \iota^*(\Delta_\P \cap \Z^N)$ is defined to be
precisely the lattice points in $\Delta_0$ which lie in the image of
$\iota^*$ of $\Delta_\P \cap \Z^N$. Thus for any $m \in W_0$, by
assumption 
there exists a preimage $\tilde{m}$ of $m$ under
$\iota^*$.

We can now define our set of sections $\sigma^m_{s,t}$. The sections depend on a choice of preimage
$\tilde{m}$ for each $m \in W_0$, but their essential properties -
such as those asserted in Theorem~\ref{theorem:main} - are independent
of these choices. For this reason and for simplicity we suppress this
choice from the notation. Specifically, we have the following. 
(It may be helpful for
the reader to refer to the diagram~\eqref{eq:big diagram for pulling
  back sections}.) 

\begin{definition}\label{definition:sigma s t}
For each $m \in W_0$, let $\tilde{m} \in \Delta_\P \cap \Z^N$ denote a
fixed choice of preimage of $m$ under $\iota^*$. Let $s, t \in \R$
with $0 \leq s < \infty$ and $0 < t \leq 1$. We define 
\begin{equation}
  \label{eq:def sections}
  \sigma^{m}_{s,t} := \tilde{\phi}^*_{1-t} \tilde{\rho}^*_{s,t}
  \tilde{\chi}^*_s \sigma^{\tilde{m}} \in H^0(X, L,
  \overline{\partial}_{s,t}). 
\end{equation}
\end{definition}

In the next section we will find an appropriate
function $t=t(s)$ so that the bases $\{\sigma^m_s :=
\sigma^m_{s,t=t(s)}\}$ will satisfy the convergence conditions asserted in
Theorem~\ref{theorem:main} with respect to the complex structures $J_s
:= J_{s,t=t(s)}$ and $\overline{\partial}_s :=
\overline{\partial}_{s,t=t(s)}$.

\section{Proof of the main theorem}\label{sec:proof}

We now proceed to a proof of the main result of this manuscript,
Theorem~\ref{theorem:main}. Much of this section is devoted to proving
results about the sections $\sigma^m_{s,t}$ defined in
Section~\ref{subsec:varying bases}, under certain hypotheses on the
parameters $s$ and $t$. At the end of this section we choose an
appropriate function $t=t(s)$ so that the sections $\sigma^m_s := \sigma^m_{s,t(s)}$
depend only on the single parameter $s$ and have the correct
convergence properties. 

We begin our discussion with a statement about supports. 
Specifically, part of the assertion of Theorem~\ref{theorem:main} is that a certain
(normalized) 
section 
weakly 
converges to a dirac-delta
function on the corresponding Bohr-Sommerfeld fiber. 
In particular,
the support of $\sigma^m_{s,t=t(s)}$ must concentrate on a
neighborhood of the Bohr-Sommerfeld fiber as $s$ gets large. We make
this precise in the proposition below, for which we need some
preliminaries. 
Let $m \in W_0 \subset \t_0^*$. For a real number
$\eta>0$, let $B_\eta(m)$ denote the open ball of radius $\eta$ around
$m$ with respect to the usual metric on $\t_0^*$. 
We introduce the following notation:
\begin{equation}\label{eq:definition B_eta}
B_\eta(m) := \mu^{-1}(B_\eta(m)), \quad 
B_{\eta,t}(m) := \mu_t^{-1}(B_\eta(m)), \quad
B_{\eta,0}(m) := \mu_0^{-1}(B_\eta(m))
\end{equation}
where $\mu, \mu_t$ and $\mu_0$ are the moment maps for the
integrable systems on $X, X_t$ and $X_0$ respectively. When the point
$m$ is clear from context, we sometimes write $B_{\eta,t} =
B_{\eta,t}(m)$, etc. 
We also let $d(vol)$ denote 
the symplectic volume form on $X_t$ and $\rho_{s,t}(X_t)$ for all
$s$ and $t$; since the
relevant maps between these spaces preserve symplectic structures, the
ambiguity in this notation does not pose problems.
We let $\sabs{\cdot}$ denote the norm with respect to the
hermitian metric on all the line bundles;
again, all relevant maps preserve the hermitian metric so there
is no ambiguity.
We let $\abs{\cdot}_{L^1(\cdot)}$ 
denote the $L^1$-norm of a section over some space;
for the sake of space and readability we will occasionally omit the
explicit mention of the space in the notation and write simply $\abs{\cdot}$.

We can now state and prove the following.

\begin{proposition}\label{prop:support on fibers}
Let $m \in W_0$ be an interior point and let $\sigma^m_{s,t} \in H^0(X, L,
\overline{\partial}_{s,t})$ be the section defined in~\eqref{eq:def sections}. 
Then there exists 
a continuous function $t' = t'(s): [0,\infty) \to [0,1]$ such that
for every $\epsilon>0$ and $\eta>0$, there exists $s_0 > 0$ such that 
\begin{equation}\label{eq:support on fibers estimate}
\int_{X \setminus B_\eta} \bigg\lvert \frac{\sigma^m_{s, t}}
    {\abs{\sigma^m_{s, t}}_{L^1(X)}} \bigg\rvert d(vol) < \epsilon
\end{equation}
for all $s > s_0$ and $0 \leq t \leq t'(s)$, and moreover, $\lim_{s
  \to \infty} t'(s) = 0$. 

\end{proposition}

\begin{remark} 
We believe the result of Proposition~\ref{prop:support on fibers} 
holds for boundary points as well as interior points but we restrict ourselves to interior points 
for the purposes of this paper.  
\end{remark} 

The proof of Proposition~\ref{prop:support on fibers} requires
several steps, the first of which states that the analogous result is
true for the 
special fiber $X_0$. We quote the following. 

\begin{lemma}\label{lemma:support-section-converges-X0}
(\cite[Proposition 6.6, (3) and (4)]{HamKon}) 
  Let $m \in W_0$ be an interior point
  and let $\tilde{m} \in \Delta_\P \cap \Z^N$ denote the
  preimage of $m$ fixed in Definition~\ref{definition:sigma s t}. 
  For any $\epsilon>0$ and $\eta>0$, there exists $s_0>0$
  such that 
\[
\int_{X_0 \setminus B_{\eta,0}} \bigg\lvert
\frac{\tilde{\chi}^*_s(\sigma^{\tilde{m}})}
{\abs{\tilde{\chi}^*_s(\sigma^{\tilde{m}})}_{L^1(X_0)}}
\bigg\rvert d(vol) < \epsilon
\]
for all $s>s_0$. 
\end{lemma}

The following estimate will also be useful. Roughly, it says that for
any fixed $s>0$ and any $\epsilon>0$, there are fibers $X_t$
of the family which are sufficiently close to $X_0$ such that the two
maps $\rho_{s,t}$ and $\phi_t$ do not differ on $X_t$ by more than distance
$\epsilon$. 

\begin{lemma}\label{lemma:rho-st-close}
 Let $s, t \in \R$ with $0 < s < \infty$ and $0 \leq t
 \leq 1$, let $\rho_{s,t}: X_t \to \P$ be the embedding in
Proposition~\ref{prop:embedding of submanifold}, let 
$\phi_t \colon X_t \to X_0$ denote the gradient-Hamiltonian flow, and
let $\epsilon>0$. Then there exists $t_0>0$ 
such that for any $x\in X_t$ with $0<t<t_0$, 
\[
d_\P(\rho_{s,t}(x), \phi_{t}(x))<\epsilon
\]
where $d_\P$ denotes the distance function on $\P$ induced from the
K\"ahler metric on $(\P, \omega_\P, J_\P)$. 
\end{lemma}

\begin{proof}
First, we show that we can choose $t$ small enough that $\phi_t(x)$ 
is close to $x$, uniformly in $X_t$. Note that this part of the argument is
independent of the parameter $s$. 
Let 
\[ 
B := \Bigl\{ (x,t) \in \X \cross [0,1] \st \pi(x) \in [0,1],\  
t\leq \pi(x) \Bigr\} 
\subseteq \P \times [0,1] \times [0,1] \subseteq \P \times \C \times [0,1].
\]
Then $B$ is a closed subset of the compact space $\P \cross [0,1]
\times [0,1]$, and is therefore 
compact.  
Consider the map $\Psi: B \to B$ given by $\Psi(x,t) = (\phi_t(x),
\pi(x)-t)$. Note that $\Psi$ is well-defined since $t \leq \pi(x)$ by
assumption so $\pi(x)-t \geq 0$. 
It follows from 
\cite[Part 1, Theorem 4.1]{HarKav} that $\Psi$ is continuous as a
function from $B$ to itself, 
and hence uniformly continous. 
In particular, this implies that for any $\delta>0$, there exists a
$t_0>0$ such that for any $t<t_0$ we have 
$d_\P(\phi_t(x), x) <\delta$ for any $x \in X_t$.

We next analyze the embeddings 
$\rho_{s,t} \colon X_t \into \mathcal{P}$.  
Recall that $\X_{[0,1]}$ denotes $\pi\inv([0,1])\subset \X$. 
For a fixed $s$, let $f_s$ denote the map 
$\X_{[0,1]} \to \P$ given by $x \mapsto \rho_{s,\pi(x)}(x)$.
Then $f_s$ is continuous by Proposition~\ref{prop:rho-smooth-in-t},
and therefore also uniformly continuous 
since $\X_{[0,1]}$ is compact.  
Recall from Proposition~\ref{prop:toric submanifolds don't move} that
$\rho_{s,0}=\id$ for all $s$.  
Thus for any $x \in X_t$ we have 
\[ f_s(\phi_t(x)) = \rho_{s,0}(\phi_t(x)) = \phi_t(x) \]
since $\phi_t(x) \in X_0$. 
Now let $\epsilon>0$ be given.  
Choose $\delta>0$ such that $d_\P(x,y)<\delta$ implies 
$d_\P(f_s(x),f_s(y))<\epsilon$ for any $x$, $y \in \X_{[0,1]}$.  
Then choose $t_0$ so that $0<t<t_0$ implies $d_\P(\phi_t(x),x)<\delta$ 
for all $x \in X_t$, as above. 
Then for all $x\in X_t$ with $0<t<t_0$, 
\[
d_\P(\rho_{s,t}(x) , \phi_t(x)) = 
d_\P(f_s(x) , f_s(\phi_t(x))) < \epsilon
\]
as required.
\end{proof}

We are now ready to prove Proposition~\ref{prop:support on fibers}.
The idea of the proof is to combine two separate estimates, as we now
sketch.  On the one hand, we will use Lemma~\ref{lemma:rho-st-close}
to argue that we can make the integral over $X$ close to the analogous
integral over $X_0$. On the other hand, we know from
Lemma~\ref{lemma:support-section-converges-X0} that the integral over
$X_0$ can be made arbitrarily small. Putting these estimates together
gives the result. We make this precise in the proof below.

\begin{proof}[Proof of Proposition~\ref{prop:support on fibers}]

We first note that by Lemma~\ref{lemma:support-section-converges-X0} 
there exists $s_0>0$ such that for any $s>s_0$ we have 
\begin{equation}\label{eq:integral on X_0}
\int_{X_0 \setminus B_{\eta,0}} \bigg\lvert
\frac{\tilde{\chi}^*_s(\sigma^{\tilde{m}})}
{\|\tilde{\chi}^*_s(\sigma^{\tilde{m}})\|_{L^1(X_0)}}
\bigg\rvert \hsm d(vol) < 
\frac{\epsilon}{2}.
\end{equation}

Next we notice that since $\tilde{\chi}_s^*(\sigma^{\tilde{m}})$ is a
holomorphic section of a hermitian line bundle,  
its norm $\sabs{\tilde{\chi}_s^*(\sigma^{\tilde{m}})}$ is a
continuous function on $\P$.
Since $\rho_{s,t}$ is continuous in $t$, as noted in
Proposition~\ref{prop:rho-smooth-in-t}, the function
$t\mapsto \abs{\tilde{\chi}_s^*\sigma^{\tilde{m}}\circ \rho_{s,t}}_{L^1(X_t)}$
is a continous function on $t$.
Define $C_0$ and $C_1$, respectively, to be the minumum and maximum
values of
$\abs{\tilde{\chi}_s^* \sigma^{\tilde{m}} \circ \rho_{s,t}}_{L^1(X_t)}$
for $t\in [0,1]$, and note that $C_0\neq 0$ for sufficiently small $t$
because 
$\tilde{\chi}_s^*(\sigma^{\tilde{m}})$ is not identically
  zero \cite[Theorem 5.3 (a5)]{HamKon}, and thus
$\tilde{\chi}_s^* \sigma^{\tilde{m}} \circ \rho_{s,t}$
is not the zero section for sufficiently small $t$. 

Similarly, 
since $\tilde{\chi}_s^*(\sigma^{\tilde{m}})$ is a holomorphic section
of a hermitian line bundle on the compact set $\P$, 
its norm $\sabs{\tilde{\chi}_s^*(\sigma^{\tilde{m}})}$ is a continuous,
and hence uniformly continuous, function on $\P$.
Thus for any $\epsilon>0$, there exists a $\delta>0$ such that if
$d_\P(x,x') < \delta$ for $x, x' \in \P$, then
\[
\bigg\lvert \sabs{\tilde{\chi}_s^*(\sigma^{\tilde{m}})}(x)
- \sabs{\tilde{\chi}^*_s(\sigma^{\tilde{m}})} (x') \bigg\rvert <
\frac{\epsilon C_0 \|\tilde{\chi}_s^* \sigma^{\tilde{m}}\|_{L^1(X_0)}}
  {4 C_1 \vol(X_t)}.
\]
Moreover, by Lemma~\ref{lemma:rho-st-close}, for a fixed $s$ with
$s>s_0$ as above, we know 
there exists 
$t_0=t_0(s)>0$ such that for any $t$ with $0<t<t_0=t_0(s)$ and any $x \in X_t$ we
have 
\[
d_\P(\rho_{s,t}(x), \phi_t(x)) < \delta.
\]
For what follows we will also choose $t_0(s)$ sufficiently small so
that $\tilde{\chi}_s^*(\sigma^{\tilde{m}}) \circ \rho_{s,t}$ is not
identically zero for $0 < t < t_0$, so in particular $C_0 > 0$. 
Now, we have that for all $x \in X_t$ with $0 < t < t_0=t_0(s)$,
\begin{equation}\label{eq:sigma-m-s-close}
\bigg\lvert 
  \sAbs{\tilde{\chi}_s^*(\sigma^{\tilde{m}})}\bigl(\rho_{s,t}(x)\bigr)
  - \sAbs{\tilde{\chi}^*_s(\sigma^{\tilde{m}})}\bigl(\phi_{t}(x)\bigr) \bigg\rvert
< \frac{\epsilon C_0 \|\tilde{\chi}_s^* \sigma^{\tilde{m}}\|_{L^1(X_0)}}
  {4 C_1 \vol(X_t)}
\leq \frac{\epsilon \|\tilde{\chi}_s^* \sigma^{\tilde{m}}\|_{L^1(X_0)}}{4 \vol(X_t)}
\end{equation}
where the last inequality is because $C_0/C_1\leq 1$.

Next we recall that the sections $\sigma^m_{s,t}$ on $X$ in
Definition~\ref{definition:sigma s t} are given by a sequence of
pullbacks. 
In particular, since both $\tilde{\phi}_{1-t}$ and
$\tilde{\rho}_{s,t}$ preserve the 
hermitian metric and $\phi_{1-t}$ preserves
symplectic structures, 
we have that 
\begin{equation}\label{eq:pullback integrals}
\int_{X \setminus B_\eta} \sabs{\sigma^m_{s,t}} \hsm d(vol) = 
\int_{X_t\setminus B_{\eta,t}} 
\sAbs{\tilde{\rho}^*_{s,t} \tilde{\chi}^*_s(\sigma^{\tilde{m}})} \hsm d(vol) =
\int_{X_t \setminus B_{\eta,t}} \sAbs{\tilde{\chi}_s^*(\sigma^{\tilde{m}})}
   \circ \rho_{s,t} \hsm d(vol), 
\end{equation}
where we also use that $\phi_{1-t}^{-1}(B_{\eta, t}) = B_\eta$ by
construction~\eqref{eq:definition mu_t} of the moment maps $\mu_t$;
for the same reason, the $L^1$-norms satisfy
\begin{equation}\label{eq:sigma^m norm equality}
  \big\| \sigma^m_{s,t} \big\|_{L^1(X)} =
  \int_{X} \sabs{\sigma^m_{s,t}} \hsm d(vol) = 
  \int_{X_t} \sAbs{\tilde{\chi}_s^*(\sigma^{\tilde{m}})} \circ \rho_{s,t}
    \hsm d(vol) =
\big\|\tilde{\chi}_s^*(\sigma^{\tilde{m}}) \circ \rho_{s,t}\big\|_{L^1(X_t)}.
\end{equation}
Similarly, since $\phi_t$ and $\tilde{\phi}_t$ preserve the
relevant structures we have 
\begin{equation}\label{eq:pullbacks to X_0 and X_t}
\int_{X_0 \setminus B_{\eta,0}} 
\sAbs{\tilde{\chi}_s^*(\sigma^{\tilde{m}})} \hsm d(vol) =
\int_{X_t\setminus B_{\eta,t}}
\sAbs{\tilde{\phi}_t^*\tilde{\chi}_s^*(\sigma^{\tilde{m}})}\hsm d(vol) =
\int_{X_t\setminus B_{\eta,t}} \sAbs{\tilde{\chi}_s^*(\sigma^{\tilde{m}})}
  \circ \phi_t \hsm d(vol).
\end{equation}
In this case, the $L^1$-norms satisfy
\begin{equation}\label{eq:norms-phi-t}
  \big\|\tilde{\chi}_s^*(\sigma^{\tilde{m}}) \big\|_{L^1(X_0)} =
  \int_{X_0} \sAbs{\tilde{\chi}_s^*(\sigma^{\tilde{m}})} \hsm d(vol)
  = \int_{X_t} \sAbs{\tilde{\chi}_s^*(\sigma^{\tilde{m}})}
  \circ \phi_t \hsm d(vol)
  = \big \| \tilde{\chi}_s^*(\sigma^{\tilde{m}}) \circ \phi_t \big\|_{L^1(X_t)};
\end{equation}
however, because $\phi_t$ preserves the symplectic structure, 
the above norms are equal for all $t\in [0,1]$.  

Because of the normalizing factors in the denominators, we will need
to show that
$\abs{\tilde{\chi}_s^* \sigma^{\tilde{m}}\circ \rho_{s,t}}_{L^1(X_t)}$
is close to
$\abs{\tilde{\chi}_s^* \sigma^{\tilde{m}}\circ \phi_t}_{L^1(X_t)}$,
which we do as follows:
For $s>s_0$ and for all $x$ in $X_t$, $0<t\leq t_0(s)$,
\begin{equation}\label{eq:estimate-on-norms}
\begin{split}
\Big\lvert \abs{\tilde{\chi}_s^* \sigma^{\tilde{m}}\circ \rho_{s,t}} 
- \abs{\tilde{\chi}_s^* \sigma^{\tilde{m}}\circ \phi_t} 
\Big\rvert    
& \leq 
\int_{X_t}
\Big\lvert \sAbs{\tilde{\chi}_s^* \sigma^{\tilde{m}}\circ \rho_{s,t}}
- \sAbs{\tilde{\chi}_s^* \sigma^{\tilde{m}}\circ \phi_t} 
\Big\rvert    \hsm d(vol) \\
& \leq 
\int_{X_t} \frac{\epsilon C_0 \|\tilde{\chi}_s^* \sigma^{\tilde{m}}\|_{L^1(X_0)}}
    {4 C_1 \vol(X_t)} \hsm d(vol)
= \frac{\epsilon C_0 \|\tilde{\chi}_s^* \sigma^{\tilde{m}}\|_{L^1(X_0)}}{4 C_1}
\end{split}
\end{equation}
where the last inequality comes from~\eqref{eq:sigma-m-s-close}.

Now, using the fact that
\[
\aabs{\frac{a}{h} - \frac{b}{k}} = \aabs{\frac{a(k-h) + h(a-b)}{hk}}
\leq \frac{\sabs{a}\sabs{k-h}}{\sabs{hk}} + \frac{\sabs{a-b}}{\sabs{k}}
\] 
we obtain that, for our fixed value of $s>s_0$ and for all $x \in X_t$
with $0<t\leq t_0(s)$, the expression 
\begin{equation*}
\left\lvert
\frac{\sabs{\tilde{\chi}_s^*\sigma^{\tilde{m}}}\bigl(\rho_{s,t}(x)\bigr)}
  {\abs{\tilde{\chi}_s^* \sigma^{\tilde{m}}\circ \rho_{s,t}}}
- \frac{\sabs{\tilde{\chi}_s^* \sigma^{\tilde{m}}}\bigl(\phi_t(x)\bigr)}
  {{\abs{\tilde{\chi}_s^* \sigma^{\tilde{m}}\circ \phi_t}}}
\right\rvert 
\end{equation*} 
is less than or equal to  
\begin{equation}\label{eq:huge-ugly-mess-with-normalizations}
\begin{split}
\frac{\sAbs{\tilde{\chi}_s^* \sigma^{\tilde{m}}\bigl(\rho_{s,t}(x)\bigr)}}
  {\abs{\tilde{\chi}_s^* \sigma^{\tilde{m}}\circ \rho_{s,t}}
  \abs{\tilde{\chi}_s^* \sigma^{\tilde{m}}\circ \phi_t}} \;
\Big\lvert 
\abs{\tilde{\chi}_s^* \sigma^{\tilde{m}}\circ \rho_{s,t}} 
- \abs{\tilde{\chi}_s^* \sigma^{\tilde{m}}\circ \phi_t}
\Big\rvert 
+  \frac{\Bigl\lvert
  \sAbs{\tilde{\chi}_s^*(\sigma^{\tilde{m}})}\bigl(\rho_{s,t}(x)\bigr)
  - \sAbs{\tilde{\chi}^*_s(\sigma^{\tilde{m}})}\bigl(\phi_{t}(x)\bigr)
  \Bigr\rvert}
{{\abs{\tilde{\chi}_s^* \sigma^{\tilde{m}}\circ \phi_t}}}.
\end{split}
\end{equation}
Using \eqref{eq:estimate-on-norms} on the first term, 
the second estimate in~\eqref{eq:sigma-m-s-close} on the second term,
the fact that 
$\sAbs{\tilde{\chi}_s^*(\sigma^{\tilde{m}})\bigl(\rho_{s,t}(x)\bigr)} \leq
\frac{C_1}{\vol(X_t)}$
by the definition of $C_1$, the fact that $C_0$ is the minimum value
of $\abs{\tilde{\chi}_s^*\sigma^{\tilde{m}} \circ \rho_{s,t}}$
(and hence
$\frac{1}{\abs{\tilde{\chi}_s^*\sigma^{\tilde{m}} \circ \rho_{s,t}}} \leq \frac{1}{C_0}$), and the equality~\eqref{eq:norms-phi-t}, this becomes  
\begin{equation}\label{eq:normalized-estimate-rho-phi-sigma}
\begin{split}
\left\lvert
\frac{\sAbs{\tilde{\chi}_s^*\sigma^{\tilde{m}}}\bigl(\rho_{s,t}(x)\bigr)}
  {\abs{\tilde{\chi}_s^* \sigma^{\tilde{m}}\circ \rho_{s,t}}}
- \frac{\sAbs{\tilde{\chi}_s^* \sigma^{\tilde{m}}}\bigl(\phi_t(x)\bigr)}
  {{\abs{\tilde{\chi}_s^* \sigma^{\tilde{m}}\circ \phi_t}}}
\right\rvert 
& \leq 
\frac{C_1}{C_0 \vol(X_t)
  \abs{\tilde{\chi}_s^* \sigma^{\tilde{m}}\circ \phi_{t}}}
\frac{\epsilon C_0 \|\tilde{\chi}_s^* \sigma^{\tilde{m}}\|}{4 C_1}
+ 
\frac{1}{{\abs{\tilde{\chi}_s^* \sigma^{\tilde{m}}\circ \phi_t}}} 
\frac{\epsilon \|\tilde{\chi}_s^* \sigma^{\tilde{m}}\|}{4 \vol(X_t)} \\
& = \frac{\epsilon}{2\vol(X_t)}.
\end{split}
\end{equation}

Putting everything together, for $s>s_0$ and $t<t_0$ for
the chosen $s_0, t_0(s)$ as above, we have 
\begin{equation}\label{eq:final estimate for supports}
\begin{split}
\int_{X \setminus B_\eta}
\frac{\sabs{\sigma^m_{s,t}}}{\abs{\sigma^m_{s,t}}} \hsm d(vol) & = 
\int_{X \setminus B_\eta}
\frac{\sabs{\sigma^m_{s,t}}}{\abs{\sigma^m_{s,t}}} \hsm d(vol)  - 
\int_{X_0 \setminus B_{\eta,0}} 
\frac{\sabs{\tilde{\chi}^*_s(\sigma^{\tilde{m}})}}
     {\abs{\tilde{\chi}^*_s(\sigma^{\tilde{m}})}}
\hsm d(vol) + 
\int_{X_0 \setminus B_{\eta,0}} 
\frac{\sabs{\tilde{\chi}^*_s(\sigma^{\tilde{m}})}}
     {\abs{\tilde{\chi}^*_s(\sigma^{\tilde{m}})}}
\hsm d(vol) \\
 & = \int_{X_t \setminus B_{\eta,t}} 
\frac{\sabs{\tilde{\chi}^*_s(\sigma^{\tilde{m}})} \circ \rho_{s,t}}
{\Abs{\tilde{\chi}_s^*(\sigma^{\tilde{m}}) \circ \rho_{s,t}}}
\hsm  d(vol) - 
\int_{X_t \setminus B_{\eta,t}} 
\frac{ \sabs{\tilde{\chi}^*_s(\sigma^{\tilde{m}})} \circ \phi_t}
     {\Abs{\tilde{\chi}^*_s(\sigma^{\tilde{m}})\circ \phi_t}}
\hsm d(vol)
 + \int_{X_0 \setminus B_{\eta,0}} 
 \frac{\sabs{ \tilde{\chi}^*_s(\sigma^{\tilde{m}}) }}
     {\Abs{\tilde{\chi}^*_s(\sigma^{\tilde{m}})}}
 \hsm d(vol) \\
& \leq \bigg\lvert 
\int_{X_t \setminus B_{\eta,t}} 
\frac{\sabs{\tilde{\chi}^*_s(\sigma^{\tilde{m}})} \circ \rho_{s,t}}
{\Abs{\tilde{\chi}_s^*(\sigma^{\tilde{m}}) \circ \rho_{s,t}}}
\hsm  d(vol) - 
\int_{X_t \setminus B_{\eta,t}} 
\frac{\sabs{\tilde{\chi}^*_s(\sigma^{\tilde{m}})} \circ \phi_t}
     {\abs{\tilde{\chi}^*_s(\sigma^{\tilde{m}})\circ \phi_t}}
\hsm d(vol) \bigg \rvert + 
\int_{X_0 \setminus B_{\eta,0}} 
 \frac{\sabs{\tilde{\chi}^*_s(\sigma^{\tilde{m}})}}
     {\abs{\tilde{\chi}^*_s(\sigma^{\tilde{m}})}}
 \hsm d(vol) \\
& \leq 
\int_{X_t \setminus B_{\eta,t}} \bigg \lvert 
\frac{\sabs{\tilde{\chi}^*_s(\sigma^{\tilde{m}})} \circ \rho_{s,t}}
{\Abs{\tilde{\chi}_s^*(\sigma^{\tilde{m}}) \circ \rho_{s,t}}} - 
\frac{\sabs{\tilde{\chi}^*_s(\sigma^{\tilde{m}})} \circ \phi_t}
     {\abs{\tilde{\chi}^*_s(\sigma^{\tilde{m}})\circ \phi_t}}
\bigg\rvert  \hsm d(vol) + 
\int_{X_0 \setminus B_{\eta,0}} 
 \frac{\sabs{\tilde{\chi}^*_s(\sigma^{\tilde{m}})}}
     {\abs{\tilde{\chi}^*_s(\sigma^{\tilde{m}})}}
 \hsm d(vol) \\
& \leq 
\int_{X_t \setminus B_{\eta,t}} \frac{\epsilon}{2 \vol(X_t)} \hsm
d(vol) + 
\int_{X_0 \setminus B_{\eta,0}} 
 \frac{\sabs{\tilde{\chi}^*_s(\sigma^{\tilde{m}})}}
     {\abs{\tilde{\chi}^*_s(\sigma^{\tilde{m}})}} \hsm d(vol)\\
& \leq \frac{\epsilon}{2 \vol(X_t)} \cdot \vol(X_t) +
\frac{\epsilon}{2} = \epsilon
\end{split}
\end{equation}
as required, where the second equality uses~\eqref{eq:pullback
  integrals}, \eqref{eq:sigma^m norm equality},
\eqref{eq:pullbacks to X_0 and X_t}, and~\eqref{eq:norms-phi-t}, 
the second-to-last inequality uses~\eqref{eq:normalized-estimate-rho-phi-sigma}, and the last inequality
uses~\eqref{eq:integral on X_0}.

Finally, we wish to prove that there exists $t'=t'(s): [0, \infty) \to
[0,1]$ a continuous function of $s$ such that~\eqref{eq:final estimate
  for supports} holds for $\sigma^m_{s,t}$ for all $s$ and all $t$
with $0 \leq t \leq t'(s)$, and also such that $t'(s) \to 0$ as $s \to
\infty$. 
To see this, notice first that 
immediately before~\eqref{eq:sigma-m-s-close} we made a choice of 
$t_0(s)$ which depended on $s$. In the subsequent argument we proved 
statements that hold 
for all $t$ with $0<t<t_0$. In particular,  
these statements are still true if we replace $t_0$ by a smaller 
choice of $t_0$.  From this it follows that we may assume without loss
of generality that $t_0 = t_0(s)$ is a monotone non-increasing
function of $s$. 
 
A standard result from real analysis (see e.g. \cite[Lemma 1.6.31(iii)]{Tao})
says that a bounded, non-decreasing monotone function $F\colon \R \to \R$ 
can be written as $F=F_c + F_{pp}$, 
where $F_c$ is a continuous monotone non-decreasing function 
and $F_{pp}$ is a jump function.  
Looking at the definition of $F_{pp}$ in the proof in~\cite{Tao},
it is obvious that $F_{pp}$ is nonnegative; 
a little more thought shows that if $F(x)$ is nonnegative
then $F_{pp}(x) \leq F(x)$ for all $x$, 
so if $F(x)$ is nonnegative then so is $F_c$.  
Taking $F(s) = t_0(-s)$ for $s\leq 0$ (and constant for $s\geq 0$), 
we can apply this decomposition to obtain that $t_0(s)$ for $s>0$ 
is the sum of 
a nonnegative jump function and 
a continuous positive non-increasing function $t'(s)$ 
which therefore satisfies $0\leq t'(s) \leq t_0(s)$.  
By shrinking $t'(s)$ if necessary, we can arrange that $t'(s)\to 0$ 
as $s\to\infty$.  
Since the inequalities at each stage are true for the chosen $s$ 
and for all $t$ with $0<t\leq t_0(s)$, 
they are true for $0<t<t'(s)$, and we are finished.
\end{proof}

% % MDH commented out, January 23 2018
% We have just seen that for any lattice
% point $m \in W_0$, the support of the section $\sigma^m_{s,t}$ converges to 
% the fiber $\mu^{-1}(m)$, provided that $s$ and $t$ lie in an
% appropriate region of the parameter space. 
% Now suppose $m \in W_0$ is an interior
% lattice point, i.e., $m \in W_0 \cap \Delta^{\circ}_0$, where $\Delta^{\circ}_0$
% denotes the interior of the polytope $\Delta_0$.

Next, the fiber $\mu^{-1}_0(m)$ is diffeomorphic to a torus and lies entirely
within the open dense torus orbit of $X_0$, and thus it is possible to
obtain much more refined information about the behavior of the family
$\{\sigma^m_{s,t}\}$ in the limit. Specifically, 
let $\Gamma(X, L^*)$ denote 
the space of \emph{smooth} (not necessarily holomorphic) sections
of the dual complex line bundle and let $\langle \cdot, \cdot \rangle$
denote the usual pairing 
between $L^*$ and $L$. For 
$\sigma \in \Gamma(X,L^*)$ we let $\| \sigma \|_{L^1(X)}$
denote its $L^1$-norm with respect to the Hermitian metric on $L$,
i.e. $\| \sigma \|_{L^1(X)} = \int_X \lvert \sigma \rvert \hsm
d(vol)$. 
We have the following. 

\begin{proposition}\label{proposition:interior} 
Let $m \in \Delta_0 \cap \Z^n$
be an interior lattice point. Let $\tau \in \Gamma(X_1, \L_1^*)$. Then 
  there exist a covariantly constant section $\delta_m$ of $(L
  \vert_{\mu^{-1}(m)}, \nabla \vert_{\mu^{-1}(m)})$, a measure
  $d\theta_m$ on $\mu^{-1}(m)$, and a continuous function $t=t(s)$ satisfying $\lim_{s \to
    \infty} t(s)=0$ such that 
\begin{equation}\label{eq:main}
\lim_{s \to \infty} \int_{X} \bigg\langle \tau, \frac{\sigma^m_{s,
    t(s)}}{\| \sigma^m_{s,t(s)}\|_{L^1(X)}} \bigg \rangle \hsm
d(vol) = 
\int_{\mu^{-1}(m)} \langle \tau, \delta_m \rangle \hsm d\theta_m. 
\end{equation}
\end{proposition} 

The idea of the proof is similar to that of
Proposition~\ref{prop:support on fibers} above and requires a
number of steps.  Namely, we will relate the LHS
of~~\eqref{eq:main} to a limit of integrals on $X_{t(s)}$ and then
approximate the integral on $X_{t(s)}$ by one on $X_0$. We then use
the fact that the analogous statement to
Proposition~\ref{proposition:interior} is already known on $X_0$; this
is the content of the following lemma.

\begin{lemma}\label{lemma:estimate on X_0}
Let $U_0$ denote the open dense $\T_0$-orbit in $X_0$ and let $m \in
\Delta_0 \cap \Z^n$ be an interior lattice point. Let $\tilde{m}$ be
the fixed choice of preimage of $m$ under $\iota^*$ as in
Definition~\ref{definition:sigma s t}. Then there exists a
covariantly constant section $\delta_{m,0}$ of $(L_\P
\vert_{\mu_0^{-1}(m)}, \nabla \vert_{\mu_0^{-1}(m)})$ over
$\mu_0^{-1}(m)$ and a measure $d\theta_{m,0}$ on $\mu_0^{-1}(m)$ such
that, for any smooth section $\tau \in \Gamma(U_0,
  L^*_\P \vert_{V_0})$ 
of the dual line bundle, we have 
\begin{equation}
\lim_{s \to \infty} \int_{X_0} 
\bigg\langle \tau, \frac{\tilde{\chi}^*_s(\sigma^{\tilde{m}})}
  {\|\tilde{\chi}^*_s(\sigma^{\tilde{m}}) \|_{L^1(X_0)}} \bigg \rangle
\hsm d(vol) = 
\int_{\mu_0^{-1}(m)} \langle \tau, \delta_{m,0} \rangle \hsm d\theta_{m,0}.
\end{equation}
\end{lemma}

\begin{proof}
This is essentially the content of an argument given in \cite{HamKon}. 
Specifically, the $V_{symp}$ and $V_{comp}$ in the proof of
\cite[Proposition 6.6(4)]{HamKon} can be identified with our $X_0$.
Similarly, their $T^\ell_\C$ (respectively $T^n_\C$) is our $\T_0$
(respectively $\T_\P$). 
Finally, to apply the argument in \cite{HamKon} it is necessary that $X_0$ is
the closure of the $\T_0$-orbit of $[1:1:\cdots:1]$ and that $\iota^*$
is surjective on lattices, which hold by our
assumptions (e) and (g), respectively, as stated in
Section~\ref{sec-main-result}. Thus, the argument of
\cite{HamKon} applies. 
\end{proof}

In order to translate the previous lemma to a statement concerning
other fibers, we need some additional information. The next lemma
recalls some results from \cite{HarKav} and also constructs
compact subsets $K_t$ which will be useful for obtaining estimates. 
Let $\Delta^{\circ}_0$ denote the interior of the moment polytope. 

\begin{lemma}\label{lemma:ut-kt}
Let $t \in [0,1]$. Then there exists an open subset $U_t \subseteq
X_t$ and a compact subset $K_t \subseteq U_t$ such that: 
\begin{enumerate}
\item\label{l:u0-smooth} 
For $t=0$, $U_0$ equals $\mu_0^{-1}(\Delta^{\circ}_0) \subseteq X_0$,
the open dense $\T_0$-orbit in $X_0$. In particular, $U_0$ lies in the
smooth locus of $X_0$.

\item\label{l:gr-H-ut} 
The gradient-Hamiltonian flow $\phi_s\colon U_t \to U_{t-s}$ 
is a diffeomorphism for all $0\leq s\leq t$.  
In particular, $\phi_t$ is a diffeomorphism from $U_t$ to $U_0$.

\item\label{l:gr-H-kt} 

  The flow $\phi_s: U_t \to U_{t-s}$ in (2) identifies $K_t$ with
$K_{t-s}$, i.e., $\phi_s(K_t) = K_{t-s}$ for all $0\leq s \leq t$.

\item\label{l:k-contains-bs} 
The subset $K_t$ contains a neighborhood of every interior
Bohr-Sommerfeld fiber. More precisely, 
there exists some $\eta>0$ such that, for any $m \in \Delta^{\circ}_0\cap \Z^n$,
the neighborhood $B_{\eta,t}(m)$ 
as in~~\eqref{eq:definition B_eta} is contained in $K_t$.
\end{enumerate}
\end{lemma}

\begin{proof}

In order to satisfy (1), we first define $U_0 :=
\mu_0^{-1}(\Delta^{\circ}_0)$. From \cite[Corollary 3.3]{HarKav} it
follows that $U_0$ is contained in the locus of points in $\X$ where
the gradient-Hamiltonian vector field is defined. Moreover, as in the
proof of \cite[Theorem 5.2]{HarKav}, we know the gradient-Hamiltonian
flow is well-defined on all of $X_t$ for any $t \neq 0$. By
\cite[Lemma 2.5]{HarKav} we may now define $U_t :=
\phi_{-t}(U_0)=\phi_t^{-1}(U_0)$ from which it is immediate that $\phi_s: U_t \to U_{t-s}$ is a
diffeomorphism from $U_t$ to $U_{t-s}$ (with inverse $\phi_{-s}$) for
any $0 \leq s \leq t$. This proves (2). 

It remains to define the compact subsets $K_t$ and to prove the claims
(3) and (4). 
Let $C \subset
\Delta^{\circ}_0$ be a connected closed subset containing within its
interior every interior
lattice point, i.e., if $m \in \Delta_0^\circ \cap \Z^n$ then $m \in
C^\circ$; such a $C$ clearly exists. We define $K_t := \mu_t^{-1}(C)$. Then $K_t$ is closed since
$\mu_t$ is continuous. Since $C \subseteq \Delta_0^\circ$ and we saw
$\mu_0^{-1}(\Delta_0^\circ) \subseteq U_0$ above, it also follows from
the definition~\eqref{eq:definition mu_t} of the $\mu_t$ that $K_t \subseteq
U_t$. Moreover, since $\mu_{t-s} \circ \phi_s = \mu_t$ by
construction of the integrable systems~\eqref{eq:definition mu_t}, and
the $\phi_t$ are diffeomorphisms on the $U_t$, it follows that
$\phi_s(K_t)=K_{t-s}$ for all $0 \leq s \leq t$. This proves
(3). Finally, since $C$ contains only a finite number of lattice
points, there exists
some $\eta>0$ 
 such that for all interior lattice points $m$, 
 the ball $B_{\eta}(m)$ is contained in $C$. 
This proves (4) and completes the proof. 
\end{proof}

Roughly, the idea in what follows is to replace the integrals in previous proofs by integrals
over $K_t$. By using Proposition~\ref{prop:support on fibers} we will
be able to control the error terms. 
Then, since $K_t$ is compact by assumption, we will be able to 
use a uniform continuity argument.  
We begin with an estimate which is uniform on $K_t$ for all
sufficiently small $t$.
This will be a key 
component of the proof of Proposition~\ref{proposition:interior}. 

\begin{lemma}\label{last-bloody-technical-argument-left}

Let $\tilde{m}$ be a preimage of an interior lattice point $m \in
\Delta_0 \cap \Z^n$. There exists a continuous function $t''=t''(s):
\R_{>0} \to [0,1]$ such that for any $\varepsilon > 0$ and any $s \in
\R_{>0}$, the following holds:  if $t \in [0,1]$ satisfies $0 \leq t
\leq t''(s)$ then 
\begin{equation}\label{eq:estimate on Kt} 
\biggl\lvert
\frac{\tilde{\rho}_{s,t}^* \tilde{\chi}_s^*(\sigma^{\tilde{m}})(x)}
{\|\tilde{\rho}_{s,t}^* \tilde{\chi}_s^*(\sigma^{\tilde{m}})\|_{L^1(X_t)}}
- \frac{\tilde{\phi}_t^* \tilde{\chi}_s^*(\sigma^{\tilde{m}})(x)}
{\|\tilde{\phi}_t^* \tilde{\chi}_s^*(\sigma^{\tilde{m}})\|_{L^1(X_t)}}
\biggr\rvert < \varepsilon.
\end{equation}
for all $x \in K_t$. 
\end{lemma}

\begin{proof}

The idea of the proof is to show that the LHS of~\eqref{eq:estimate on
  Kt} is (the norm of) a continuous section of a line bundle over a suitable family that
is equal to zero for all $x \in K_0$ (i.e. when $t=0$) for any value
of $s$; then we use uniform continuity.

Recall that $\X_{[0,1]}$ denotes the restriction of our family $\X$ to
the subset $[0,1] \subseteq \C$, i.e. $\X_{[0,1]} := \pi^{-1}([0,1])
\subset \X$, and
$\mathcal{U}_{[0,1]} = \{ x \in \X_{[0,1]} \st x \in U_{\pi(x)}\}$
is the family of open dense subsets $U_t$ in each fibre,
as in~\eqref{eq:X-and-U-fam}.
We also define 
\begin{align*}
\K_{[0,1]} := \{ x \in \X_{[0,1]} \st x \in K_{\pi(x)}\}
\end{align*}
to be the family of the compact sets $K_t$
from Lemma~\ref{lemma:ut-kt} over ${[0,1]}$.
Let $L_\K$ (respectively $L_{\mathcal{U}}$) denote $L_\X \vert_{\K_{[0,1]}}$
(respectively $L_\X \vert_{\mathcal{U}_{[0,1]}}$).

Now fix $s \in \R$ with $s>0$ and let $x \in U_t \subset X_t$ for any
$t$. Recall that $\phi_t$ is the gradient-Hamiltonian flow, so in
particular $\phi_t: X_t \to X_0$ takes $X_t$ to $X_0$ and $U_t$ to
$U_0$ (cf. Lemma~\ref{lemma:ut-kt}). In the LHS of~\eqref{eq:estimate
  on Kt}, the expression $\tilde{\phi}_t^*
\tilde{\chi}_s^*(\sigma^{\tilde{m}})(x)$ is by definition equal to
$\big( \tilde{\phi}_t^{-1} \circ \tilde{\chi}_s^*(\sigma^{\tilde{m}})
\vert_{X_0} \circ \phi_t \big)(x)$. In particular,
$\tilde{\phi}_t^*\tilde{\chi}_s^*(\sigma^{\tilde{m}})$ sends each
$U_t$ to $L_t$ by definition, so it is a section of the bundle $L_t \to U_t$.
Putting these together for all $U_t$ for $t \in [0,1]$, we
obtain a section of $L_{\mathcal{U}} \to \mathcal{U}_{[0,1]}$.
The $L^1$-norm in the denominator is independent of $t$, 
since it can be written as 
\[ 
\|\tilde{\phi}_t^* \tilde{\chi}_s^*(\sigma^{\tilde{m}})\|_{L^1(X_t)}
= \int_{X_t} \sAbs{ \tilde{\phi}_t^* \tilde{\chi}_s^*(\sigma^{\tilde{m}})} \hsm d(vol)
= \int_{X_0} \sAbs{\tilde{\chi}_s^*(\sigma^{\tilde{m}})} \hsm d(vol)
\]
since $\phi_t$ preserves the Hermitian and symplectic structures,
and so the denominator of the second term is constant on $\mathcal{U}_{[0,1]}$.

Recall that $\sigma^{\tilde{m}}$ is holomorphic (hence smooth),
$\chi_s$ is smooth, and the map $\phi: \mathcal{U}_{[0,1]} \to U_0$
given by $x \mapsto \phi_{\pi(x)}(x)$ is smooth (since it is the flow
of a smooth vector field). Together with the last statement in
Lemma~\ref{lemma:lift-grH} we may conclude that the section of
$L_{\mathcal{U}} \to \mathcal{U}_{[0,1]}$ obtained above (sending $x
\in U_t$ to $\tilde{\phi}_{\pi(x)}^*
\tilde{\chi}_s^*(\sigma^{\tilde{m}})(x)$) is smooth. Finally,
recalling that when $t=0$ the maps $\phi_0: U_0 \to U_0$ and
$\tilde{\phi}_0: L_0 \vert_{U_0} \to L_0 \vert_{U_0}$ are both equal
to the identity, we conclude that the above section restricts on $U_0$
to be equal to
$\frac{\tilde{\chi}_s^*(\sigma^{\tilde{m}})}
{\abs{\tilde{\chi}_s^*(\sigma^{\tilde{m}})}}\Big\vert_{U_0}$.

Next, we consider the expression $\tilde{\rho}_{s,t}^*
\tilde{\chi}_s^*(\sigma^{\tilde{m}})$ contained in the LHS
of~\eqref{eq:estimate on Kt}. Recall that $\rho_{s,t}$ is the
embedding $X_t \to \P$ given in~\eqref{def:rho s t} and specified in
Proposition~\ref{prop:embedding of submanifold}, and $\tilde{\rho}_{s,t}$ is the lift of
$\rho_{s,t}$ to the line bundles (cf. Lemma~\ref{lemma:lift-rhos}). In particular,
for $x \in U_t$ the expression 
$\tilde{\rho}_{s,t}^*
\tilde{\chi}_s^*(\sigma^{\tilde{m}})(x)$ is by definition equal to 
$\big( \tilde{\rho}_{s,t}^{-1} \circ
\tilde{\chi}_s(\sigma^{\tilde{m}}) \vert_{\rho_{s,t}(U_t)} \circ
\rho_{s,t} \big)(x)$. As in the above case, by construction
$\tilde{\rho}_{s,t}^* \circ \tilde{\chi}_s^*(\sigma^{\tilde{m}})$
sends each $U_t$ to $L_t$ by definition, so it is a section of $L_t
\to U_t$ and by putting these together we obtain a section of
$L_{\mathcal{U}} \to \mathcal{U}_{[0,1]}$. 
Similarly to the above case, the last statement of
Proposition~\ref{prop:rho-tilde-smooth-in-t} allows us to conclude
that this section is continuous.
In this case the $L^1$-norm in the denominator is no longer independent
of $t$, but it will be continuous in $t$, and never zero since the
section is not the zero section.  
Moreover, by our assumption on
$\{\chi_s\}$ and Proposition~\ref{prop:toric submanifolds don't move}
as well as Proposition~\ref{prop:rho-tilde-smooth-in-t} we also know
that $\rho_{s,0}=\id$ for all $s$ and its lift $\tilde{\rho}_{s,0}$ 
acts as the identity on $L_0
\vert_{U_0}$, so 
$\tilde{\rho}_{s,0}^*\tilde{\chi}_s^*(\sigma^{\tilde{m}}) = \tilde{\chi}_s^*(\sigma^{\tilde{m}})$.

Now for $x\in\mathcal{K}_{[0,1]}$, let 
\[ h_s(x) =
\frac{\tilde{\rho}_{s,t}^* \tilde{\chi}_s^*(\sigma^{\tilde{m}})(x)}
{\abs{\tilde{\rho}_{s,t}^* \tilde{\chi}_s^*(\sigma^{\tilde{m}})}}
- \frac{\tilde{\phi}_t^* \tilde{\chi}_s^*(\sigma^{\tilde{m}})(x)}
{\abs{\tilde{\phi}_t^* \tilde{\chi}_s^*(\sigma^{\tilde{m}})}}.
\]
By the preceding discussion, $h_s$ is a continuous section of 
$L_{\K}$ over $\K_{[0,1]}$ and in particular its absolute value
$|h_s|$ is a continuous function from $\K_{[0,1]}$ to $\R$. 
Moreover, by the above, $h_s(x) = 0$ for $x\in X_0$.

By the continuity of $h_s$, and since we know $K_0 \subseteq
h_s^{-1}(0)$, we may cover $K_0$ with open sets $U_\alpha$ with the
property that $|h_s(x)| < \varepsilon$ for all $U_\alpha$. Extend this
to an open cover of $\K_{[0,1]}$. The compactness of $\K_{[0,1]}$
implies that there exists a finite subcover, a subset of which is a
finite cover of the subset $K_0$ of $\K_{[0,1]}$. In particular, from
the construction of the cover, we may conclude that there exists some
$t''(s)>0$ such that for all $0 \leq t < t''(s)$ and $x \in K_t$ we have
$\sabs{h_s(x)}<\varepsilon$. Now an argument similar to that in the proof
of Proposition~\ref{prop:support on fibers} shows that in fact we may
choose $t''(s)$ to be a continuous function of $s$, as claimed. 
\end{proof}

We are finally ready for a proof of
Proposition~\ref{proposition:interior}.

\begin{proof}[Proof of Proposition~\ref{proposition:interior}]
We first define the covariantly constant section $\delta_m$ and
measure $d\theta_m$. Recall that the integrable system on the variety $X$ is
defined by pulling back that on $X_0$ via $\phi_1$. In particular
$\phi_1$ induces a pullback of the action-angle coordinates on $U_0$
to $U=U_1=\phi_1^{-1}(U_0)$ and a diffeomorphism of tori $\mu_0^{-1}(m) \cong \mu^{-1}(m)
\cong (S^1)^n$. Hence by using $\phi_1$ we may pull back the
covariantly constant $\delta_{m,0}$ and measure $d\theta_{m,0}$ on
$U_0$ of Lemma~\ref{lemma:estimate on X_0} to a covariantly constant
$\delta_m$ and $d\theta_m$ respectively. 
Let $\tau \in \Gamma(U, \L^*_{\P} \vert_{U})$ be a test section. 
For simplicity of notation we
let $(\phi_1)_*\tau$ denote the section of $L^*_\P \vert_{U_0}$
obtained by using the identifications $\phi_1: U=U_1 \to
U_0$ and $\tilde{\phi}_1: L_\P \vert_{U} \to L_\P \vert_{U_0}$, so
$(\phi_1)_*\tau (x):= \tilde{\phi}_1 \circ \tau \circ \phi_1^{-1}(x)$. 
Then by definition we have 
\[
\int_{\mu_0^{-1}(m)} \langle (\phi_1)_* \tau, \delta_{m,0} \rangle
\hsm d\theta_{m,0} = \int_{\mu^{-1}(m)} \langle \tau, \delta_m \rangle
\hsm d\theta_m
\]
for any test section. 
Now let $\epsilon>0$. From the above it suffices to prove that we can find a continuous
function $t=t(s)$ and an $s_0>0$ such that for all $s>s_0$ 
we have 
\begin{equation}\label{eq:test function integral estimate}
\bigg\lvert \int_X \bigg\langle \tau, \frac{\sigma^m_{s,
    t(s)}}{\abs{\sigma^m_{s,t(s)}}} \bigg \rangle \hsm d(vol)
- 
\int_{\mu_0^{-1}(m)} \langle (\phi_1)_* \tau, \delta_{m,0} \rangle
\hsm d\theta_{m,0} \bigg\rvert < \epsilon.
\end{equation}
First, we know from Lemma~\ref{lemma:estimate on X_0} that there
exists $s_1$ such that for all $s>s_1$ we have 
\begin{equation}
\bigg\lvert\int_{X_0} \bigg\langle (\phi_1)_* \tau,
\frac{\tilde{\chi}_s^*(\sigma^{\tilde{m}})}
  {\abs{\tilde{\chi}_s^*(\sigma^{\tilde{m}})}} \bigg\rangle \hsm d(vol) 
- \int_{\mu_0^{-1}(m)} \langle (\phi_1)_* \tau, \delta_{m,0} \rangle
\hsm d\theta_{m,0} \bigg\rvert
< \frac{\epsilon}{4}.
\end{equation}
Moreover, since $\phi_1 = \phi_t \circ \phi_{1-t}$ and all maps
preserve the relevant structures, we have 
\begin{equation}\label{eq:one of the pieces of triangle inequality}
\bigg\lvert \int_{X_t} \bigg\langle (\phi_{1-t})_* \tau,
\frac{\tilde{\phi}^*_t \tilde{\chi}^*_s(\sigma^{\tilde{m}})}
  {\abs{\tilde{\phi}^*_t \tilde{\chi}^*_s(\sigma^{\tilde{m}})}}
  \bigg \rangle \hsm d(vol) 
- \int_{\mu_0^{-1}(m)} \langle (\phi_1)_* \tau, \delta_{m,0} \rangle
\hsm d\theta_{m,0} \bigg\rvert < \frac{\epsilon}{4}.
\end{equation}
Next note that 
$\tau$ is a smooth section
on the compact space $X$, so there
exists a constant $C>0$ such that $\abs{\tau}<C$ 
for all $x \in X$. 
Since $\phi_{1-t}$ preserves the Hermitian structure, 
this also implies that $\sabs{(\tilde{\phi}_{1-t})_* \tau} <C$ 
for all $x\in X_t$ and for all $t$. 
Let $K=K_1$ be the compact subset of $X=X_1$ from
Lemma~\ref{lemma:ut-kt}. Our next step is to approximate the integrals
over $X$ and $X_0$ with integrals over $K$ and $K_0$.  
Specifically, following Lemma~\ref{lemma:ut-kt}~(4) 
let $\eta>0$ be such that 
such that $B_\eta = B_{\eta}(m) \subset K$. 
By Proposition~\ref{prop:support on fibers}, there exists $s_2>0$ such
that
\[ 
\int_{X\smallsetminus B_{\eta}} \bigg\lvert{
  \frac{\sigma^m_{s,t}}{\abs{\sigma^m_{s,t}}} \bigg\rvert \hsm d(vol) 
< \frac{\epsilon}{4C}}
\]
for any $s>s_2$ and $0\leq t\leq t'(s)$ where $t=t'(s)$ is the function
constructed in the proof of Proposition~\ref{prop:support on
  fibers}. 
Since $B_\eta \subset K$ and because we have an upper bound
$\sabs{\tau}<C$ on the norm of $\tau$, we conclude
\[ \int_{X\smallsetminus K} \bigg\lvert{\frac{\sigma^m_{s,t}}{\abs{\sigma^m_{s,t}}}} \bigg\rvert \sabs{\tau} < 
\int_{X\smallsetminus B_\eta} \bigg\lvert{\frac{\sigma^m_{s,t}}{\abs{\sigma^m_{s,t}}}} \bigg\rvert \sabs{\tau} 
< \frac{\epsilon}{4}
\]
for $s>s_2$ and $0\leq t\leq t'(s)$. 
Thus 
\begin{equation}\label{eq:K approximates X}
\begin{split}
\aabs{\int_X \bigg\langle \tau, \frac{\sigma^m_{s,t}}{\abs{\sigma^m_{s,t}}} \bigg\rangle \hsm d(vol) - 
\int_K \bigg\langle \tau, \frac{\sigma^m_{s,t}}{\abs{\sigma^m_{s,t}}} \bigg\rangle \hsm d(vol)} & = 
\aabs{\int_{X\smallsetminus K} \bigg\langle \tau, \frac{\sigma^m_{s,t}}{\abs{\sigma^m_{s,t}}} \bigg\rangle \hsm d(vol)} \\
& \leq 
\int_{X\smallsetminus K} \sabs{\tau}\bigg\lvert\frac{\sigma^m_{s,t}}{\abs{\sigma^m_{s,t}}} \bigg\rvert d(vol)
< \frac{\epsilon}{4}
\end{split}
\end{equation}
for these choices of $s$ and $t$, and in this sense,
the integral over $X$ is well approximated by one over $K$.

A similar argument, using Lemma~\ref{lemma:support-section-converges-X0}
applied to $X_0$ and $K_0$,
gives an $s_3$ such that for $s>s_3$ 
\begin{equation}\label{eq:K_0 approximates X_0}
\aabs{\int_{X_0} \bigg\langle (\phi_1)_* \tau,
\frac{\tilde{\chi}_s^*(\sigma^{\tilde{m}})}
  {\abs{\tilde{\chi}_s^*(\sigma^{\tilde{m}})}} \bigg\rangle \hsm d(vol) 
- \int_{K_0} \bigg\langle (\phi_1)_* \tau,
\frac{\tilde{\chi}_s^*(\sigma^{\tilde{m}})}
  {\abs{\tilde{\chi}_s^*(\sigma^{\tilde{m}})}} \bigg\rangle \hsm d(vol) }
< \frac{\epsilon}{4}
\end{equation}
and so the integral over $X_0$ is well approximated by one over $K_0$.

Next, for any $0<t<1$  we can push forward by $\phi_{1-t}$ to 
rewrite the integral over $K$ as an integral over $K_t$ as
follows. Recalling the definition of $\sigma^m_{s,t}$
from~\eqref{eq:def sections}
we have 
\[
\int_K \bigg\langle \tau,
\frac{\sigma^m_{s,t(s)}}{\abs{\sigma^m_{s,t(s)}}} \bigg \rangle \hsm d(vol) 
= \int_{K_t} \bigg\langle (\phi_{1-t})_* \tau, 
\frac{\tilde{\rho}_{s,t}^*\tilde{\chi}_s^*(\sigma^{\tilde{m}})}
  {\abs{\tilde{\rho}_{s,t}^*\tilde{\chi}_s^*(\sigma^{\tilde{m}})}}
  \bigg\rangle \hsm d(vol).
\]
We then have the following: 
\begin{equation}\label{eq:last big sequence}
\begin{split}
\bigg\lvert
\int_{K_t} \bigg\langle (\phi_{1-t})_* \tau,
  \frac{\tilde{\rho}_{s,t}^*\tilde{\chi}_s^*(\sigma^{\tilde{m}})}
  {\abs{\tilde{\rho}_{s,t}^*\tilde{\chi}_s^*(\sigma^{\tilde{m}})}}
  \bigg\rangle \hsm d(vol) 
& -   % this & is just so the next lines don't run off the page
\int_{K_t} \bigg\langle (\phi_{1-t})_* \tau,
  \frac{\phi_t^*\tilde{\chi}_s^*(\sigma^{\tilde{m}})}
  {\abs{\phi_t^*\tilde{\chi}_s^*(\sigma^{\tilde{m}})}}
  \bigg\rangle \hsm d(vol)
\bigg\rvert  \\
& \leq 
\int_{K_t} \bigg\lvert \bigg\langle (\phi_{1-t})_* \tau,
  \frac{\tilde{\rho}_{s,t}^*\tilde{\chi}_s^*(\sigma^{\tilde{m}})}
  {\abs{\tilde{\rho}_{s,t}^*\tilde{\chi}_s^*(\sigma^{\tilde{m}})}} \bigg\rangle
- \bigg\langle (\phi_{1-t})_* \tau,
  \frac{\phi_t^*\tilde{\chi}_s^*(\sigma^{\tilde{m}})}
  {\abs{\phi_t^*\tilde{\chi}_s^*(\sigma^{\tilde{m}})}}
  \bigg\rangle \bigg\rvert \hsm d(vol) 
\\
& =
\int_{K_t} \bigg\lvert \bigg\langle (\phi_{1-t})_* \tau,
  \frac{\tilde{\rho}_{s,t}^*\tilde{\chi}_s^*(\sigma^{\tilde{m}})}
  {\abs{\tilde{\rho}_{s,t}^*\tilde{\chi}_s^*(\sigma^{\tilde{m}})}}
- \frac{\phi_t^*\tilde{\chi}_s^*(\sigma^{\tilde{m}})}
  {\abs{\phi_t^*\tilde{\chi}_s^*(\sigma^{\tilde{m}})}}
  \bigg\rangle \bigg\rvert \hsm d(vol) 
\\
& \leq 
\int_{K_t} \sAbs{(\phi_{1-t})_* \tau}
\bigg\lvert \frac{\tilde{\rho}_{s,t}^*\tilde{\chi}_s^*(\sigma^{\tilde{m}})}
  {\abs{\tilde{\rho}_{s,t}^*\tilde{\chi}_s^*(\sigma^{\tilde{m}})}}
- \frac{\phi_t^*\tilde{\chi}_s^*(\sigma^{\tilde{m}})}
  {\abs{\phi_t^*\tilde{\chi}_s^*(\sigma^{\tilde{m}})}}
\bigg\rvert \hsm d(vol) \\
& \leq 
C \int_{K_t} \bigg\lvert
\frac{\tilde{\rho}_{s,t}^*\tilde{\chi}_s^*(\sigma^{\tilde{m}})}
  {\abs{\tilde{\rho}_{s,t}^*\tilde{\chi}_s^*(\sigma^{\tilde{m}})}}
- \frac{\phi_t^*\tilde{\chi}_s^*(\sigma^{\tilde{m}})}
  {\abs{\phi_t^*\tilde{\chi}_s^*(\sigma^{\tilde{m}})}}
\bigg\rvert \hsm d(vol) \\
\end{split}
\end{equation}
where the last inequality again uses the upper bound
$\sabs{(\phi_{1-t})_*\tau}<C$. Now let $\varepsilon =
\frac{\epsilon}{2C \vol(X_t)}$. (Note that volumes are equal for all
fibers, i.e. $\vol(X)=\vol(X_t)$ for all $t$.) Applying
Lemma~\ref{last-bloody-technical-argument-left} to this value of
$\varepsilon$, we obtain a continuous monotone (non-increasing)
function $t''(s)$ of $s$ such that for all $0 < t < t''(s)$ and all $x
\in K_t$ we have 
\begin{equation}\label{eq:estimate from previous lemma}
  \bigg\lvert
\frac{\tilde{\rho}_{s,t}^*\tilde{\chi}_s^*(\sigma^{\tilde{m}})}
  {\abs{\tilde{\rho}_{s,t}^*\tilde{\chi}_s^*(\sigma^{\tilde{m}})}}
- \frac{\phi_t^*\tilde{\chi}_s^*(\sigma^{\tilde{m}})}
  {\abs{\phi_t^*\tilde{\chi}_s^*(\sigma^{\tilde{m}})}}
  \bigg\rvert
< \frac{\epsilon}{4C\vol(X_t)}
\end{equation}
which implies that the integral in~\eqref{eq:last big sequence}
is less than $\epsilon/4$.

Finally, let $t(s) = \min\{t'(s), t''(s)\}$ be the minimum of 
the two continuous functions $t'(s)$ and $t''(s)$ defined earlier.  
Then $t(s)$ is a continuous, positive, decreasing function of $s$, 
and the estimates in ~\eqref{eq:K approximates X}
and~\eqref{eq:last big sequence} 
hold for all $0<t<t(s)$.
Let $s_0 = \max\{s_1, s_2, s_3\}$.  
The triangle inequality then implies that the LHS
of~\eqref{eq:test function integral estimate} is
less than or equal to the sum of the left-hand sides 
of~\eqref{eq:one of the pieces of triangle inequality},
\eqref{eq:K approximates X}, \eqref{eq:K_0 approximates X_0},
and~\eqref{eq:last big sequence},
each of which is less than $\epsilon/4$ for $s>s_0$.
This implies that, for this function $t(s)$ and this choice of $s_0$, 
the inequality~\eqref{eq:test function integral estimate} holds for $s>s_0$, 
and we are finished with the proof of Proposition~\ref{proposition:interior}.
\end{proof}

In Section~\ref{ss:varying-cplx-str} 
we constructed for each $s$ and $t$ a complex structure $J_{s,t}$ 
on $X$ that is compatible with the symplectic structure. Recall also
that 
the parameter $s \in [0,\infty)$ corresponds to the 
deformation of complex structures
while the $t \in [0,1]$ parameter corresponds to the gradient-Hamiltonian flow 
from $X_1$ to $X_0$. 
Our final task is to  
construct the family $J_s$ of complex structures and the basis of
sections $\sigma^m_s$ in the statement of our main
Theorem~\ref{theorem:main}. 
The basic idea behind the definition below is to let $s$ go to $\infty$ and $t$ go to $0$
simultaneously 
in such a way that 
the convergence which is claimed in Theorem~\ref{theorem:main} occurs.  
Indeed, in Proposition~\ref{proposition:interior} we constructed a
continuous function $t(s)$ of $s$ which gives rise to certain
key estimates for any $s>0$ and $0 \leq t < t(s)$. 
Thus, setting $J_s := J_{s,t(s)}$ would (nearly) do
the job; however, 
although we need $J_0$ to be the
original complex structure (which is $J_{0,1}$) on $X$, our
construction of $t(s)$ does not
guarantee that $t(0)=1$. 
The solution to this problem 
is, roughly speaking, to first move along the gradient-Hamiltonian flow to $t_0$ 
while keeping $s=0$ before ``turning on'' the other deformation. 
More precisely, we make the following definition. 

\begin{definition}\label{def:J_s}
Let $t(s)$ denote the continuous function constructed in
Proposition~\ref{proposition:interior}. For $s \in [0,\infty)$, we define 
\[
J_s := 
\begin{cases}
J_{0,1+(t_0-1)s} & \textup{ if } 0\leq s \leq 1 \textup{ and }  \\
J_{s-1,t(s-1)}  & \textup{ if } s > 1.
\end{cases}
\]
\end{definition}
Note that by construction $J_0 = J_{0,1}$ is the original complex structure on $X$, 
and as $s\to \infty$, $J_s$ has the same convergence properties as $J_{s,t(s)}$.
Moreover, by construction, the family $J_s$ is continuous with respect
to the parameter $s$. 
We are finally ready to prove our main
Theorem~\ref{theorem:main}.

\begin{proof}[Proof of Theorem~\ref{theorem:main}]
Consider the family $\{J_s\}_{s \in [0,\infty)}$ given in
Definition~\ref{def:J_s}. As already noted above, $\{J_s\}$ is a
continuous family and $J_0=J_{0,1}$ is the original complex structure
by construction. By Lemma~\ref{lemma:pairs Kahler}, the pair
$(\omega=\omega_1, J_s)$ is a K\"ahler structure on $X$ for each $s
\in [0,\infty)$. Now define 
\begin{equation}\label{eq:def sigma m s}
\sigma^m_s := 
\begin{cases} 
\sigma^m_{0,1+(t_0-1)s} & \textup{ if } 0 \leq s \leq 1 \\
\sigma^m_{s-1,t(s-1)} & \textup{ if } s > 1.
\end{cases}
\end{equation}
By construction, $\sigma^m_s$ is an element of $H^0(X, L,
\overline{\partial}_s)$. Moreover, by definition $J_s$ and
$\sigma^m_s$ have the same limiting properties as $J_{s,t(s)}$ and
$\sigma^m_{s,t(s)}$, and for any interior point $m \in W_0$ it was
shown in Proposition~\ref{proposition:interior} that $\sigma^m_{s,t(s)}$ has the required
limiting property of~\eqref{eq:equation in main theorem} in the statement of
Theorem~\ref{theorem:main}. Thus it remains only to ensure that for
each fixed $s$ w have a basis of $H^0(X, L,
\overline{\partial}_s)$. By the limiting properties of the
$\sigma^m_s$, we know that as $s$ goes to $\infty$, the supports of 
the sections $\sigma^m_s$ are increasingly concentrated in pairwise
disjoint neighborhoods. It follows that the set 
$\{\sigma^m_s\}$ must be linearly independent for $s>s_0$ with $s_0$
sufficiently large. Since
the complex manifolds $(X,J_s)$ and holomorphic line bundles $(L,
\overline{\partial}_s)$ are isomorphic for all $s$, we also know
that $\dim H^0(X, L, \overline{\partial}_s)$ is constant
for all $s$, and in particular by assumption (h) in the statement of
Theorem~\ref{theorem:main}, we have $\dim
H^0(X,L,\overline{\partial}_s) = |W_0|$ for all $s$. Thus for $s \geq
s_0$ the set $\{\sigma^m_s\}$ is linearly independent and also spans, so it is a basis of
$H^0(X,L,\overline{\partial}_s)$ as desired. For $0 \leq s \leq s_0$,
following \cite[Section 7.2]{HamKon} we extend the basis
$\{\sigma^m_s\}_{m \in W_0}$, $s \geq s_0$, to bases of $H^0(X,L,
\overline{\partial}_s)$ for $s$ satisfying $0 \leq s \leq s_0$ in a way that
preserves the continuity in the parameter $s$. This family then
satisfies all the required properties. 
\end{proof}

\section{Toric degenerations coming from valuations and 
Newton-Okounkov bodies}\label{sec:NOBY}

In this section we will show that the toric degenerations coming 
from Newton-Okounkov bodies as in~\cite{HarKav} 
can be used to create many examples of algebraic varieties $X$ with prequantum
data $(\omega, J, L, h, \nabla)$ which satisfy the hypotheses of
Theorem~\ref{theorem:main}. This is the content of the main result of
this section, Theorem~\ref{theorem:toric deg from NOBY}. 

We first very briefly recall the ingredients in the definition of a
Newton-Okounkov body. For details we refer the reader to \cite{KavKho,
  LazMus} and also \cite{HarKav}.  We begin with the definition of a
valuation (in our setting). We equip $\Z^n$ with a group ordering
e.g. a lexicographic order.

\begin{definition}\label{definition:valuation}
\begin{enumerate}
  \item   Let $A$ be a $\C$- algebra. A {\bf valuation} on $A$   is a function  \[\nu:A\setminus\{0\}\rightarrow \Z^n \] satisfying the following:
        \begin{enumerate}
          \item  $ \nu(c f)=\nu(f)$ for all $f\in A\setminus\{0\}$ and $c\in\C\setminus\{0\}$,
          \item $\nu(f+g) \geq \min\{\nu(f),\nu(g)\}$ for all $f, g \in A
  \setminus \{0\}$ with $f+g \neq 0$.
        
        \item $\nu(fg)=\nu(f) + \nu(g)$ for all $f,g \in A \setminus \{0\}$.
        \end{enumerate}

 \item The image $\nu(A\setminus\{0\})\subset\Z^n$  of a valuation $\nu$ on a $\C$-algebra $A$  is clearly a semigroup
and is called the \emph{value semigroup} of the pair
$(A,\nu)$.
 \item Moreover, if in addition the valuation has the property
that for any pair $f,g\in A\setminus\{0\}$ with same value
  $\nu(f)=\nu(g)$ there exists a non-zero constant
  $c \neq 0 \in \C$ such that  either $\nu(g-c f)>\nu(g)$ or else
  $g-c f = 0$
then we say that  the valuation has \emph{one-dimensional leaves}.
\end{enumerate}

\end{definition}

If $\nu$ is a valuation with one-dimensional leaves, then the image of
$\nu$ is a sublattice of $\Z^n$ of full rank.  Hence, by replacing
$\Z^n$ with this sublattice if necessary, we will always assume
without loss of generality that $\nu$ is surjective.

Given a variety $X$, there exist many possible valuations with
one-dimensional leaves on its field of rational functions 
$\C(X)$. Strictly speaking, we do not need detailed knowledge of the
construction in this paper, but we note for the reader's reference
that the following example is the one which arises naturally 
in geometric contexts: 
for an $n$-dimensional variety $X$, a choice of an (ordered) coordinate
system at a smooth point $p$ on $X$ gives a valuation on $\C(X)$ with
one-dimensional leaves, essentially by computing the order of the zero
or pole with respect to the coordinates. See
e.g. \cite[Examples 2.12 and 2.13 ]{KavKho} or \cite{LazMus} for details. 

The following proposition is simple but fundamental \cite[Proposition 2.6]{KavKho}: 
\begin{proposition} \label{prop-val-dim}
Let $\nu$ be a valuation on $\C(X)$ with one-dimensional leaves. Let $V \subset \C(X)$ be a finite-dimensional subspace of $\C(X)$.
Then $\dim_\C(V) = \sAbs{\nu(V \setminus \{0\})}$. 
\end{proposition}

Let $X$ be a projective variety of dimension $n$ over $\C$ equipped
with a very ample line bundle $L$.  Let $E
:= H^0(X, L)$ denote the space of global sections of $L$; it is a
finite dimensional vector space over $\C$. The line bundle $L$ gives
rise to the \emph{Kodaira map $\Phi_E$ of $E$}, 
from $X$ to the projective space $\p(E^*)$. 
The assumption that $L$ is very ample implies that the Kodaira map $\Phi_E$ is an embedding.

Now let $E^k$ denote the image of the $k$-fold product $E \otimes
\cdots \otimes E$ in $H^0(X, L^{\otimes k})$ under the natural map
given by taking the 
product of sections. (In general this map may not be surjective.) The
homogeneous coordinate ring of $X$ with respect to the 
embedding $\Phi_E: X \hookrightarrow \p(E^*)$ can be identified with the graded algebra 
$$R = R(E) = \bigoplus_{k \geq 0} R_k,$$ where $R_k := E^k$. This is a subalgebra of the {\it ring of sections} 
$$R(L) = \bigoplus_{k \geq 0} H^0(X, L^{\otimes k}).$$

For a fixed $\nu$ we now associate a semigroup $S(R) \subset \N \times \Z^n$
to $R$. 
First we identify $E = H^0(X, L)$ with a (finite-dimensional)
subspace of $\C(X)$ by choosing a non-zero element $h \in E$ and
mapping $f \in E$ to the rational function $f/h \in \C(X)$. Similarly,
we can associate the rational function $f/h^k$ to an element $f \in R_k := E^k \subseteq H^0(X,
L^{\otimes k})$. We define 
\begin{equation}\label{eq:definition S}
S = S(R) = S(R,\nu,h) = \bigcup_{k > 0} \{ (k, \nu(f / h^k)) \mid f \in
E^k \setminus \{0\}\}.
\end{equation}
If $f \in R_k = E^k$ is a homogeneous element of degree $k$ we also define:
$$\tilde{\nu}(f) = (k, \nu(f/h^k)).$$

Now define
$C(R) \subseteq \R \times \R^n$ to be the cone generated by the
semigroup $S(R)$, i.e., it is the smallest closed convex cone centered
at the origin containing $S(R)$. We can now define the central object
of interest.

\begin{definition}\label{definition:NO}
  Let $\Delta=\Delta(R)=\Delta(R,\nu)$ be the slice of the cone $C(R)$ at level 1, that is, $C(R)\cap (\{1\}\times \R^n)$,
  projected to $\R^n$ via the projection to the second factor $\R
  \times \R^n \to \R^n$. We have
\[
\Delta = \overline{ \textup{conv} \left( \bigcup_{k>0}
    \left\{\frac{x}{k} :  (k,x) \in S(R) \right\} \right) }.
\]
The convex body $\Delta$ is called the \emph{Newton-Okounkov body of
  $R$} with respect to the valuation $\nu$.
\end{definition}

From now on, we place
the additional assumption that: 
\begin{center}  
\emph{$S$ is finitely generated.} 
\end{center} 
The above assumption is a rather strong condition on $(X, L, \nu)$ but
it holds in many cases of importance.  We note that it is possible to
have a finitely generated semigroup $S$ for one choice of a valuation
$\nu$ and a non-finitely generated one for a different choice of
$\nu$. From the above assumption it follows that the Newton-Okounkov
body $\Delta(R)$ is a rational polytope.  In this context, Anderson
proved the following \cite[Corollary 5.3]{Anderson}.

\begin{theorem} \label{th-toric-degen-NO-body}
There is a flat family $\pi: \X \to \C$ such that:
\begin{itemize}
\item[(a)] For any $z \neq 0$ the fiber $X_z = \pi^{-1}(z)$ is
  isomorphic to $X$, and 
$\pi^{-1}(\C^*)$ is isomorphic to $X \times \C^*$. For the remainder
of the discussion 
we fix an isomorphism $X \times \C^* \to \pi^{-1}(\C^*) \subset \X$.
\item[(b)] The special fiber $X_0 = \pi^{-1}(0)$ is isomorphic to
  $\textup{Proj}(\textup{gr} R) \cong \textup{Proj}(\C[S])$ and is equipped with an action of
  $\T = (\C^*)^n$, where $n=\dim_\C X$. 
The normalization of the variety 
$\textup{Proj}(\textup{gr} R)$ is the toric variety associated to the rational polytope $\Delta(R)$.
\end{itemize}
\end{theorem}

The explicit construction of the family $\X$ in \cite{Anderson}
depends on a choice of a so-called Khovanskii basis
$\mathcal{B} = \{f_{ij}\}$ (cf. \cite[Definition 8.1]{HarKav}, and
also see
\cite{KavMan} for a general theory of Khovanskii bases). The set
$\mathcal{B}$ also allows us to concretely embed $\X$ as a subvariety
of $\P \times \C$ for an appropriate ``large'' projective space
$\P$. Some of the details are relevant for our later discussions so we
briefly recall the construction here; for details we refer to
\cite[Sections 8 and 9]{HarKav}.

By assumption the semigroup $S \subset \N \times \Z^n$ is finitely generated. So we can find a finite set consisting of homogeneous elements in $R$ such that their valuations are a set of generators for $S$. More precisely, let $r>0$ be a positive integer and let $\mathcal{B} = \{f_{ij}\}$, for $1 \leq i \leq r$, $1 \leq j \leq n_i = \dim(R_i)$, be a finite set of elements in $R$ satisfying the following properties:
\begin{enumerate} 
\item[(a)] the $f_{ij}$ are homogeneous, with $f_{ij} \in R_i$ for all $1 \leq i \leq r, 1 \leq j \leq n_i$, and 
\item[(b)] for each $i$, the collection $\{f_{i1}, f_{i2}, \ldots, f_{i n_i}\}$ is a vector space basis for $R_i$, 
\item[(c)] the set of images $\tilde{\nu}(\mathcal{B}) = \{\tilde{\nu}(f)
\mid f \in \mathcal{B}\}$ generate $S=S(R)$, 
\end{enumerate}
%It follows from our assumption that $S$ is finitely generated that
%such a $\mathcal{B}$ exists, and 
For the remainder of this discussion we fix this ``Khovanskii basis''
$\mathcal{B}$. 

We now describe more explicitly, in terms of the Khovanskii basis, the
toric degeneration $\X$ constructed in \cite{Anderson} and a
concrete embedding of $\X$ into $\P \times \C$ for a suitable large
projective space $\C$. 
Let $S_d := S \cap (\{d\} \times \Z^n)$ denote the level-$d$ piece of
the semigroup $S$. By Proposition~\ref{prop-val-dim}, $\dim(E^d) = 
|\nu(E^d)| = |S_d|$, and since the $\{f_{ij}\}$ form a Khovanskii
basis, for each $s \in S_d$ we know there
exists some monomial $f_{11}^{\alpha_{11}} f_{12}^{\alpha_{12}} \cdots
  f_{r n_r}^{\alpha_{r n_r}}$ in the $\{f_{ij}\}$'s, where
  $\sum_{i=1}^r i \sum_{j=1}^{n_i} \alpha_{ij} = d$, such that $\nu(f_{11}^{\alpha_{11}} f_{12}^{\alpha_{12}} \cdots
  f_{r n_r}^{\alpha_{r n_r}}) = s$. So for each $s \in S$ we fix a
  choice of such exponents $\alpha_s := (\alpha_{(ij),s})$ such that
  the above holds. Then the set 
\begin{equation}\label{eq:basis for E^d}
\{ f_{11}^{\alpha_{(11),s}} f_{12}^{\alpha_{(12),s}} \cdots
  f_{r n_r}^{\alpha_{(r n_r),s}} \mid s \in S_d\}
\end{equation}
 forms a basis for
  $E^d$. In \cite{Anderson}, a collection of integers $w_{ij}$ are
  associated to the $f_{ij}$ in a certain way (for details see
  \cite{Anderson} and \cite[Section 8]{HarKav}). Using these integers
  $w_{ij}$ and the above choices we can describe explicitly the toric
  degeneration $\X$ and its embedding as follows. We first define a
  morphism $X \times \C^* \to \p((E^d)^*) \times \C^*$ by
  expressing the Kodaira embedding $X \to \p((E^d)^*)$ explicitly
  using the above basis for $E^d$. In coordinates
  we have 
\begin{equation}\label{eq:embedding of family}
(x,t) \mapsto \bigg( \Big( t^{\sum_{ij}w_{ij} \alpha_{ij}} f_{11}(x)^{\alpha_{11}} \cdots
    f_{rn_r}(x)^{\alpha_{rn_r}} \; \bigg\lvert \; \alpha_{ij} = \alpha_{(ij), s} \; , \; s \in S_d \Big), t\bigg)
\end{equation}
%\begin{equation}\label{eq:embedding of family}
%(x,t) \mapsto \bigg( \Big( t^{\sum_{ij}w_{ij} \alpha_{ij}} f_{11}(x)^{\alpha_{11}} \cdots
%    f_{rn_r}(x)^{\alpha_{rn_r}} \; \bigg\lvert \; \sum_{i=1}^r i
%  \bigl( \sum_{j=1}^{n_i} \alpha_{ij} \bigr) = d \; , \; \alpha_{ij}
%  \in \Z_{\geq 0} \Big), t\bigg)
%\end{equation}

Then the toric
degeneration $\X \subseteq \P \times \C$ is defined to be the closure
of the image of~\eqref{eq:embedding of family}. By its construction,
$X$ is isomorphic to the fiber $X_1$ and $X_0$ is a toric variety
\cite[Corollary 5.3]{Anderson}. 

Note that the pullback to $X$ of the line bundle $L_\P$ over $\P$ is
$L^{\otimes d}$ by construction. Given any prequantum data
$(\omega_\P, L_\P, h_\P, \nabla_\P)$ on $\P = \p((E^d)^*)$, it is
clear that this data
can be pulled back via the embedding~\eqref{eq:embedding of family} to
prequantum data $(\omega, L^{\otimes d}, h, \nabla)$ on the line
bundle $L^{\otimes d}$ over $X$. 

We have the following, which is the main result of this section. 

\begin{theorem}\label{theorem:toric deg from NOBY}
Let $X$ be a smooth, irreducible complex algebraic variety with
$\dim_{\C}(X) = n$, let $L$ be a very ample line bundle on $X$, and
$E:=H^0(X,L)$. Then there exists a sufficiently large positive integer
$d$ and prequantum data $(\omega_\P, h_\P, \nabla_\P)$ on $L_\P \to
\P$ such that the family $\X \subseteq \P \times \C$ constructed above
is a toric degeneration of $X$ in the sense of
Section~\ref{sec-main-result}, and moreover, with respect to the
pullback prequantum data $(\omega, h, \nabla)$ on
$L^{\otimes d} \to X$, this toric degeneration satisfies all the required
hypotheses (a)-(h), and thus gives 'convergence of polarization' in these cases. 
\end{theorem}

\begin{proof}[Proof of Theorem~\ref{theorem:toric deg from NOBY}]
  The fact that $\X \subseteq \P \times \C$ is a toric degeneration
  satisfying the hypotheses (a) and (b) of Theorem~\ref{theorem:main}
  follows from the construction in \cite{Anderson} and is shown in
  \cite{HarKav}. Moreover, in \cite[Section 9]{HarKav} an appropriate
  K\"ahler structure $\omega_\P$ on $\P$ is constructed which
  satisfies condition (d). Indeed, the construction of $\omega_\P$ is
  by pulling back a Fubini-Study form associated to a hermitian
  structure on an (even larger) projective space, and in particular --
  by pulling back the standard prequantum data on a projectivization
  of a vector space equipped with a hermitian structure -- it is clear
  that we can construct the prequantum data $(\omega_\P, h_\P,
  \nabla_\P)$ compatible with $\omega_\P$. Now the hypothesis (c) 
  follows by construction, since the relevant prequantum data are
  defined via pullbacks.

We also claim that hypothesis (e) holds in our situation. Indeed, by
our choice of basis~\eqref{eq:basis for E^d} of $E^d$, it follows that
our embedding~\eqref{eq:embedding of family} has the property that
the coordinates of the embedding correspond exactly to elements of
$S_d$. In particular, it follows from \cite[Proposition 5.1]{Anderson} and
\cite[Section 8]{HarKav} that the special fiber $X_0$ is exactly the
closure of a $\mathbb{T}_0$-orbit through a point of the form
$[1:1:\cdots:1]$, where $\mathbb{T}_0 \cong (\C^*)^n$ acts by weight $s \in S_d
\subseteq \Z^n$ on the coordinate associated to $s \in S_d$. 
It also follows that the set $W_0$ defined in~\eqref{eq:def W_0} is
precisely $S_d \subseteq \Z^n$, and thus hypothesis (h) follows from
the fact, already observed above, that $\dim E^d = \dim H^0(X,L^{\otimes d}) =
|S_d|$. 

Next, it is well-known
\cite[Chapter II, Section 5, Exercise 5.14(b)]{Hartshorne} that for
sufficiently large $d \gg 0$, we have $E^d = H^0(X,L^{\otimes d})$, so in
particular $\p((E^{d})^*) = \p((H^0(X,L^{\otimes {d}}))^*)$ and the restriction map
$H^0(\P,L_\P) = H^0(\p(H^0(X,L^{\otimes {d}})^*), \mathcal{O}(1)) \to H^0(X,
L^{\otimes {d}})$ is surjective, so $\X$ satisfies hypothesis (f)
for sufficiently large $d$.

Finally, we claim that for sufficiently large $d \gg 0$ we also have that
hypothesis (g) holds. 
By definition of the torus $\mathbb{T}_0$, the graded semigroup $S$ generates
$\Z \times M$ where $M$ is the character lattice of $\mathbb{T}_0$
(and can be identified with $(\t^*_0)_{\Z}$). It is easy 
to see that in this situation there exists $d$ sufficiently large such
that $S_d = \nu(H^0(X,L^{\otimes d}) \setminus \{0\})$ generates $M$. 
Since $S_d$ is contained in 
$\iota^*((\t^*_{\P})_{\Z})$ by what we said above, the hypothesis (g)
follows for $d$ sufficiently large. 

Thus, by taking $d$ large enough so that both of the last two phenomena occur,
we obtain the results claimed in the theorem. 
\end{proof}

Finally, in the situation of an integrable system coming from a
toric degeneration arising from a Newton-Okounkov body,
the sections $\sigma^m_{s,t}$ we construct in \S\ref{subsec:varying bases}
form a basis of $H^0(X,L,\bar{\del}_{s,t})$ for all values of $s$ and $t$.
In particular, in this case
we do not need to ``extend the basis'' non-constructively
as in the last sentence of the proof of Theorem~\ref{theorem:main},
at the very end of \S\ref{sec:proof}.

\begin{theorem} \label{thm-sections-lin-ind}
Let the notation be as in Theorem~\ref{theorem:toric deg from NOBY}. 
Then for any fixed $s\geq 0$, $t\neq 0$, 
the set $\{\sigma^m_{s,t} \st m \in W_0\}$
constructed in Definition~\ref{definition:sigma s t} is linearly independent.
In particular, the set $\{ \sigma^m_s \st m\in W_o \}$ constructed
in~\eqref{eq:def sigma m s} is a basis for $H^0(X,L^d,\bar{\del}^{s})$
for every value of $s\geq 0$.
\end{theorem}

\begin{proof}      
%Recall that the section $\sigma^m_{s,t}$ is constructed as follows:
%Given a lattice point $m\in W_0$,
%find a lattice point $\tilde{m}$ in $\Delta_P\cap \Z^N$.  
%This is associated to a canonical section in $H^0(\P,L_\P)$,
%which we denote $\sigma^{\tilde{m}}$;
%it can be described as $z^{\tilde{m}}$ in a suitable trivialization
%and coordinates.  
%Then $\sigma^m_{s,t}$ is the restriction of $\tilde{\chi_s}^*\sigma^{\tilde{m}}$ to
%$\rho_{s,t}(X_t) \subset \P$, pulled back to $X$ by
%%$\tilde{\rho}_{s,t}$ and
%$\tilde{\phi}_t$.

By Definition~\ref{definition:sigma s t}, and since $\tilde{\phi}_{1-t}$ is an isomorphism of line bundles
%(see Lemma~\ref{lemma:chi-tilde} and
(see Lemma~\ref{lemma:lift-grH}), 
it suffices to show that the $\tilde{\rho}_{s,t}^* \tilde{\chi}^*_{s} \sigma^{\tilde{m}}$ are linearly independent in $H^0(X_t, L_t)$. 
%
% restrictions of the $\sigma^{\tilde{m}}$
%to $\rho_{s,t}(X_t)$ are linearly independent.
%
%
By~\cite[Proposition 6.1(2)]{HamKon},
for each $s\geq 0$ and each $0 < t \leq 1$ 
we have a diffeomorphism $\underline{\chi}_{s,t}\colon X_t \to X_t$
such that the following diagram commutes:
\[ 
\xymatrix{ 
\P \ar[r]^{\chi_s} & \P \\
X_t \ar@{^{(}->}[u]^{\rho_{s,t}} \ar[r]^{\underline{\chi}_{s,t}} & X_t \ar@{^{(}->}[u]_\rho \\
}
\]
where $\rho$ is the standard embedding of $X_t$ into $\P$ as a complex
manifold, % (see Proposition~\ref{prop:embedding of submanifold}),
and such that $(X_t, \rho^*_{s,t} \omega_\P, \chi^*_s J_\P)$ is \kahler. (Note that with our identifications the map $\chi_0$ in \cite{HamKon} is the identity.) The $\rho_{s,t}$ appearing in the diagram above are the embeddings of Proposition~\ref{prop:embedding of submanifold}. 
Similarly, by~\cite[Proposition 6.3(2)]{HamKon}
the map $\underline{\chi}_{s,t}$ lifts to a map $\tilde{\underline{\chi}}_{s,t} \colon L_t \to L_t$
such that 
\[
\xymatrix{ 
L_\P \ar[r]^{\tilde{\chi}_{s}} & L_\P \\
L_\P \vert_{X_t} \ar[r]^{\tilde{\underline{\chi}}_{s,t}} \ar[u]^{\tilde{\rho}_{s,t}}
& L_\P \vert_{X_t} \ar[u]_{\tilde{\rho}}
}
\]
commutes, where $\tilde{\rho}_{s,t}$ are the maps in Lemma~\ref{lemma:lift-rhos}.
Then $\tilde{\chi}_s^* \sigma^{\tilde{m}}\vert_{X_t} =
\tilde{\underline{\chi}}_{s,t}^* (\sigma^{\tilde{m}}\vert_{X_t})$. Since the $\tilde{\chi}_{s,t}$ are line bundle isomorphisms, 
% Since  $\tilde{\rho}_{s,t}$ is also an isomorphism of line bundles
% (see Lemma~\ref{lemma:lift-rhos} and
% Proposition~\ref{prop:rho-tilde-smooth-in-t},
it will suffice to show that the $\sigma^{\tilde{m}} \vert_{X_t}$'s are linearly independent.

To see this, recall the construction of the embedding of 
$X\cross \C^*$ into $\P\cross \C$.  
The sections 
\[ \{f_s = f_{11}^{\alpha_{(11),s}} f_{12}^{\alpha_{(12),s}} \cdots
f_{r n_r}^{\alpha_{(r n_r),s}} \mid s \in S_d\} \]
from \eqref{eq:basis for E^d} form a basis for the
global sections of $\mathcal{O}(1) = L_\P$ on the projective space $\P =
\mathbb{P}((E^d)^*)$. The embedding of the family $X \times
\C^* \subset \X$ in $\P \times \C$ in this basis is given by the map
in \eqref{eq:embedding of family} whose components are
$t^{\sum_{ij}w_{ij} \alpha_{ij}} f_s$, $s \in S_d$. 
Since the line bundle is $\mathcal{O}(1)$, 
the functions $\sigma^{\tilde{m}}$ in the basis of its holomorphic sections 
are simply the coordinate functions on $\P$, 
and so the section $\sigma^{\tilde{m}}\vert_{X_t}$ corresponds to one of the coordinates 
$t^{\sum_{ij}w_{ij} \alpha_{ij}} f_s$.  

By construction, for fixed $t \neq 0$ the values of the valuation
$\tilde{v}$ on the components $t^{\sum_{ij}w_{ij} \alpha_{ij}} f_s$
are distinct.  Since elements with
distinct values of valuation are linearly independent, 
% Similarly, for $t=0$ we
% have an embedding of the toric variety $X_0$ in $\P =
% \mathbb{P}((E^d)^*)$ which is a monomial map written in the above
% basis $\{f_s \mid s \in S_d\}$.
% [MDH: Not sure we need to consider $t=0$.]
  %The values of the valuation on the components of this embedding
  %also correspond to distinct $s \in S_d$.
it follows that, for any fixed $t \neq 0$, the sections $\{ \sigma^{\tilde{m}} \vert_{X_t} \st m \in W_0 \}$ 
 are linearly independent, as desired.  
\end{proof}

% %% KIUMARS' ORIGINAL WORDING:
% \begin{remark}   \label{rem-sections-lin-ind}
% Recall from \eqref{eq:basis for E^d} that the sections
% $$\{f_s = f_{11}^{\alpha_{(11),s}} f_{12}^{\alpha_{(12),s}} \cdots
% f_{r n_r}^{\alpha_{(r n_r),s}} \mid s \in S_d\}$$ form a basis for the
% global sections of $\mathcal{O}(1)$ on the projective space $\P =
% \mathbb{P}((E^d)^*)$. Moreover, the embedding of the family $X \times
% \C^* \subset \X$ in $\P \times \C$ in this basis is given by the map
% in \eqref{eq:embedding of family} whose components are
% $t^{\sum_{ij}w_{ij} \alpha_{ij}} f_s$, $s \in S_d$. We note that, by
% construction, for fixed $t \neq 0$ the values of the valuation
% $\tilde{v}$ on the components $t^{\sum_{ij}w_{ij} \alpha_{ij}} f_s$
% are distinct corresponding to $s \in S_d$. Similarly, for $t=0$ we
% have an embedding of the toric variety $X_0$ in $\P =
% \mathbb{P}((E^d)^*)$ which is a monomial map written in the above
% basis $\{f_s \mid s \in S_d\}$.
%   %The values of the valuation on the components of this embedding
%   %also correspond to distinct $s \in S_d$.
% It follows that, for any fixed $t$, the sections $\{f_s \mid s \in
% S_d\}$ restricted to $X_t \subset \P \times \{t\} = \P$ are linearly
% independent. This is a consequence of the fact that elements with
% distinct values of valuation are linearly independent.
% From this, one
% can show that, in the notation of Section \ref{subsec:varying bases},
% for any fixed $s, t$, the sections $\sigma^m_{s,t}$, $m \in W_0$, are
% linearly independent.
% \end{remark}

\end{document}